\documentclass[11pt,leqno]{amsart}
\usepackage[T1]{fontenc}
\usepackage{lmodern}
\usepackage{mathtools, amsfonts, delarray, amssymb, amsmath, amsthm, a4, a4wide}
\usepackage{esint, cite}

\usepackage{mathrsfs}
\usepackage{comment}
\usepackage{graphicx}
\usepackage{latexsym}
\usepackage{epsfig}
\usepackage{color}
\usepackage{microtype}
\usepackage[colorlinks=true,linkcolor=blue,citecolor=blue,urlcolor=blue]{hyperref}

\mathtoolsset{showonlyrefs}
 
\numberwithin{equation}{section}
\newtheorem{theorem}{Theorem}[section]
\newtheorem{proposition}[theorem]{Proposition}
\newtheorem{lemma}[theorem]{Lemma}
\newtheorem{defn}[theorem]{Definition}
\newtheorem{corollary}[theorem]{Corollary}
\theoremstyle{remark}
\newtheorem{remark}[theorem]{Remark}

\newcommand{\e}{\varepsilon}
\newcommand{\p}{\partial}
\newcommand{\R}{\mathbb R}
\newcommand{\F}{\mathbf F}
\newcommand{\dd}{\,\mathrm{d}x}
\newcommand{\dH}{\,\mathrm{d}\mathcal H^{n-1}}
\newcommand{\osc}{\operatorname{osc}}
\newcommand{\diam}{\operatorname{diam}}

\newcommand{\cof}{\operatorname{cof}}
\newcommand{\diver}{\operatorname{div}}
\newcommand{\loc}{\mathrm{loc}}
\newcommand{\norm}[1]{\left\lVert#1\right\rVert}

\allowdisplaybreaks[1]
\title[Linearized Monge--Amp\`ere equations in divergence form with applications ]{Global H\"older estimates for linearized Monge--Amp\`ere equations in divergence form with applications to dual semigeostrophic equations and periodic homogenization}
\author{Guoqing Cui}
\address{School of Mathematical Sciences, Peking University, Beijing 100871, China}
\email{gqcui25@stu.pku.edu.cn}
\author{Chong Gu}
\address{Department of Mathematics, Indiana University,
Bloomington, IN 47405, USA}
\email{chongu@iu.edu}
\author{Nam Q. Le}
\address{Department of Mathematics, Indiana University,
Bloomington, IN 47405, USA}
\email{nqle@iu.edu}

\author{Ling Wang}
\address{Department of Decision Sciences and BIDSA, Bocconi University, Milano, Italy}
\email{ling.wang@unibocconi.it}

\author{Bin Zhou}
\address{School of Mathematical Sciences, Peking University, Beijing 100871, China}
\email{bzhou@pku.edu.cn}

\thanks{Guoqing Cui and Bin Zhou were  partially supported by  National Key R$\&$D Program of China 2020YFA0712800, 2023YFA009900 and NSFC  Grant 12271008. 
Chong Gu and Nam Q. Le were supported in part by the National Science Foundation under grant DMS-2452320. 
 Ling Wang was funded by the European Union: the European Research Council (ERC), through StG “ANGEVA”, project number: 101076411. Views and opinions expressed are however those of the authors only and do not necessarily reflect those of the European Union or the European Research Council. Neither the European Union nor the granting authority can be held responsible for them.}
\date{}

\begin{document}
\begin{abstract}
We establish interior and global H\"older estimates for linearized Monge--Amp\`ere equations in divergence form in all dimensions $n\geq 3$, when the Hessian determinant of the convex Monge--Amp\`ere potential is bounded above and below by positive constants and the vector field on the right-hand side is bounded. The estimates use the $L^p$ norm of the solution for any $p>1$. A key ingredient is an $L^1$ estimate, with appropriate decay rates, for the gradient of the Green's function of the linearized Monge--Amp\`ere operator in small sections, uniform with respect to the pole. As applications, we obtain uniform H\"older estimates for the time derivatives of the primal and dual potentials in the three-dimensional periodic dual semigeostrophic system when the initial density is bounded away from zero and infinity, and establish a linear convergence rate for periodic homogenization of the Monge--Amp\`ere equation.
\end{abstract}
\maketitle

\section{Introduction and statement of the main results}\label{sec:intro}

In this paper, we establish interior and global H\"older estimates for linearized Monge--Amp\`ere equations in divergence form 
in all dimensions $n\geq3$ where the convex Monge--Amp\`ere potential
has Hessian determinant  bounded between two
positive constants and the vector field on the right-hand side is bounded.  We apply our results to obtain uniform H\"older estimates for the time
derivatives of the primal and dual potentials in the three-dimensional
periodic dual semigeostrophic system when the initial density is bounded
away from zero and infinity, and to establish a linear convergence rate in periodic homogenization for the Monge--Amp\`ere equation in all dimensions.

\medskip

Let $\Omega\subset\R^n$ be a convex domain, and let $u\in C^3(\Omega)$ be
convex with
\begin{equation}\label{eq:det}
 0<\lambda\leq\det D^2u\leq\Lambda\qquad\text{in }\Omega.
\end{equation}
Throughout, the cofactor matrix of the Hessian matrix $D^2 u= (D_{ij}u)_{1\leq i, j\leq n}$ is denoted by \[U=(U^{ij})_{1\leq i, j\leq n}=\cof D^2u=(\det D^2u)(D^2u)^{-1}.\] 
It is positive
definite and divergence-free; see \cite[Lemma~2.56]{Le24}. The cofactor matrix comes from linearizing the Monge--Amp\`ere operator: \[U=\frac{\p\det D^2 u}{\p (D^2 u)}.\]

\medskip

 We are interested in obtaining interior and global H\"older estimates for the linearized Monge--Amp\`ere equations in divergence form 
\begin{equation}\label{eq:pde}
 L_uv=\diver\F
\end{equation}
with a bounded vector field $\F$, where $L_u$ is the linearized Monge--Amp\`ere operator associated with the convex Monge--Amp\`ere potential $u$:
\[
 L_uv:=\diver(U Dv)=U^{ij}D_{ij}v.
\]

This type of equation appears in singular Abreu's equation \cite{KLWZ, Le_CPAM, Le23, LZ} and semigeostrophic equations in meteorology \cite{ACDF, ACDF14, BB, Cu06, CF, CuG, CNP, Le18, Lo05, LoN} to be described in the next subsection. For more information on the linearized Monge--Amp\`ere equations, we refer to \cite{Le24} and \cite{GL26} and the references therein.

\subsection{H\"older estimates for linearized Monge--Amp\`ere equations in divergence form}

Caffarelli and Guti\'errez \cite{CG97} proved a fundamental interior Harnack inequality for the homogeneous linearized Monge--Amp\`ere equation \eqref{eq:pde} with $\F=0$
 under \eqref{eq:det}. Their result is an affine invariant version of the classical Harnack inequality of Moser, and Krylov and Safonov for linear, uniformly elliptic equations in divergence form and nondivergence form, respectively. 
 Due to the affine invariance of the linearized Monge--Amp\`ere equations, Euclidean balls were replaced by sections of the Monge--Amp\`ere potential function in \cite{CG97}.
Sections are sublevel sets of convex functions after subtracting their supporting hyperplanes.

\medskip
 For \eqref{eq:pde} with a bounded vector field $\F$, Le \cite{Le18} established an interior H\"older estimate in dimension
two. The higher-dimensional result when $n\geq 3$ in \cite[Theorem~15.6]{Le24} assumes
an additional integrability condition on the Hessian $D^2 u$; Le asked in \cite[Remark 15.10]{Le24} if this additional condition can be removed. Wang
\cite[Theorem~1.5]{Wang25} obtained an interior H\"older estimate  for \eqref{eq:pde} under
$(D^2u)^{1/2}\F\in L^q$, $q>n$, and asked in
\cite[Remark~1.6(2)]{Wang25} whether boundedness of $\F$ alone suffices.
Related estimates for linearized Monge--Amp\`ere equations with drifts were obtained by Kim \cite{Kim25}.

\medskip

Cui, Wang, and Zhou \cite{CWZ25} developed interior H\"older estimates for $L_u v=\mu$ with measures $\mu$ having
controlled growth on sections. Their application to bounded vector fields
requires the divergence to have one sign; see their Theorem~1.4 and
Remark~1.5. Gu and Le \cite{GL26} used estimates for the
Monge--Amp\`ere gradient $(D^2u)^{-1/2}D_xg$ of the Green's function of the linearized Monge--Amp\`ere operator $L_u$ to give a unified, different proof of the interior H\"older estimates in \cite{CWZ25, Wang25}. They also obtained 
a global version.  Here, we treat arbitrary bounded vector fields in all dimensions $n\geq 3$.

\medskip

Our first result is concerned with interior H\"older estimates.

\begin{theorem}[Interior H\"older estimate with bounded vector fields]
\label{thm:main}
Let $\Omega\subset\R^n$ $(n\geq 3)$ be a convex domain, and let $u\in C^3(\Omega)$ be a
convex function satisfying \[0<\lambda\leq \det D^2 u\leq \Lambda<\infty.\]
 Fix $x_0\in\Omega$ and $h_0>0$ such that
\begin{equation}\label{eq:outer}
 S_{4h_0}\Subset\Omega,\qquad\text{where }
 S_t:=S_u(x_0,t)=\big\{x \in \overline{\Omega}: u(x) < u(x_0) + Du(x_0)\cdot (x - x_0) + t\big\}.
\end{equation}
Let $p\in(1,\infty)$, $\F\in L^\infty(S_{2h_0};\R^n)$, and
$v\in W^{1,2}_{\loc}(S_{2h_0})\cap L^p(S_{2h_0})$ satisfy
\eqref{eq:pde} weakly, that is,
\[
 \int_{S_{2h_0}} U Dv\cdot D\varphi\dd
 =\int_{S_{2h_0}}\F\cdot D\varphi\dd
 \qquad\text{for all }\varphi\in C_c^\infty(S_{2h_0}).
\]
There exist $\gamma=\gamma(n,\lambda,\Lambda)\in(0,1)$ and
$C=C(n, p,\lambda,\Lambda,h_0, \diam(S_{4h_0}))>0$ such that $v$ has a representative with
\begin{equation}\label{eq:main}
 \norm{v}_{L^\infty(S_{h_0})}+[v]_{C^\gamma(S_{h_0})}
 \leq C\left(\norm{\F}_{L^\infty(S_{2h_0})}
                   +\norm{v}_{L^p(S_{2h_0})}\right).
\end{equation}
\end{theorem}

Theorem~\ref{thm:main} answers the question in \cite[Remark 15.10]{Le24} and 
\cite[Remark~1.6(2)]{Wang25} and removes the extra Hessian-integrability
condition from the higher-dimensional part of
\cite[Theorem~15.6]{Le24}. As stated, the constant in
\eqref{eq:main} does not depend on the ellipticity ratio of the cofactor matrix $U$ on an
interior compact set.

\medskip

In Theorem \ref{thm:finite-Lq-holder}, we extend the interior H\"older estimates in Theorem \ref{thm:main} to the case of $L^q$ vector fields where $q$ is large enough, depending only on $n,\lambda$, and $\Lambda$.
\medskip

An important ingredient in the proof of Theorem~\ref{thm:main} is the $L^1$ estimate, with appropriate decay rates, for the gradient of the Green's function of the linearized Monge--Amp\`ere operator on small sections. For previous studies on 
the Green's function, we refer the reader to \cite{GL26, Le16, Le17, Le18, Le20, Le24, Mal13, Mal14, Mal17, TiW}.

\begin{proposition}[$L^1$ estimate for the gradient of the Green's function]
\label{prop:green}
Assume $n\geq3$, \eqref{eq:det} and \eqref{eq:outer}. Let $S=S_u(z,h)$ satisfy
$S_u(z,2h)\Subset S_{3h_0/2}$. Let $g_S(\cdot, y)$ be the Green's function of the linearized Monge--Amp\`{e}re operator $L_u=D_i(U^{i j} D_j)$ in $S$ with pole $y\in S$, 
that is, 
 \[
                    -D_i\big(U^{i j} D_j g_S(\cdot, y)\big)  =\delta_y \quad \text { in } S, \quad
	               g_S(\cdot, y)  =0 \quad \text { on } \p S.
            \]
            Then, there is a constant $\beta=\beta(n,\lambda,\Lambda)\in (0, 1)$ such that
\begin{equation}\label{eq:green}
 \sup_{y\in S}\int_S|D_xg_S(x,y)|\dd\leq Ch^\beta,
 \qquad C=C(n,\lambda,\Lambda,h_0, \diam(S_{4h_0})).
\end{equation}
\end{proposition}
\begin{remark}
The estimate \eqref{eq:green} is uniform even when the pole approaches $\partial S$.
Its proof uses the geometry of sections of the Monge--Amp\`ere equation \cite[Chapter~5]{Le24} and the Green's function estimates in
\cite[Chapter~14]{Le24}. We give here some heuristic ideas. We estimate the 
Euclidean gradient $D g$ of the Green's function $g:=g_S(\cdot, y)$ via the Monge--Amp\`ere gradient $(D^2 u)^{-1/2}  Dg$ with a possibly large prefactor being the Laplacian of $u$ (see \eqref{DgDelta}):
\[|Dg|^2\leq \lambda^{-1}(\Delta u) U^{ij}D_i gD_j g\equiv \lambda^{-1}(\Delta u) (\det D^2 u) |(D^2 u)^{-1/2} Dg|^2.\]
The whole section $S$ is foliated by sublevel sets $S_k=\{g\leq k\}$ of the Green's function. It is interesting to note that $U^{ij}D_i gD_j g$ gives total mass $k$ on $S_k$, and this can be large. However, this is compensated by a localization property which says that the level-band $S_k\setminus S_{k/2}$ is
contained in a section of $u$ with height of order  $k^{-\frac{2}{n-2}}$ (see \eqref{eq:levels}); furthermore, on this section, the Laplacian measure $\Delta u\,\mathrm{d}x$  can be shown to have total measure of order $k^{-\frac{2}{n-2} (\frac{n-2}{2} + \frac{2\alpha}{1+\alpha})}$ (see \eqref{eq:mass}) where
$\alpha=\alpha(n,\lambda,\Lambda)\in (0, 1)$ is the exponent in the $C^{1,\alpha}$ estimates for $u$. Via the H\"older inequality, this indicates that the total mass of $|Dg|$ in $S_k\setminus S_{k/2}$ decays when $k$ is large. A dyadic summation over level-bands gives the desired result.
\end{remark}
We also obtain global results under suitable assumptions on the domain and boundary data.
In particular, a global extension  of Theorem \ref{thm:main} is the following.

\begin{theorem}[Global H\"older estimate with $L^q$ vector fields]
\label{thm:global-finite-Lq}
Let $\Omega$ be a uniformly convex domain in $\R^n$ $(n\geq 3)$, that is, for all $z\in \partial \Omega$, there is a ball $B_R(z_0)$ such that $\Omega \subset B_R(z_0)$ and $\partial \Omega \cap \partial B_R(z_0) = \{z\}$ for some uniform convexity radius $R>0$.  
Let $\partial \Omega \in C^3$ and $u\in C^{1, 1}(\overline{\Omega})\cap C^3(\Omega)$ be a convex function satisfying \[\lambda \leq \det D^2 u\leq \Lambda \quad \text{and} \quad u|_{\partial \Omega} \in C^3.\]
Let $q>q_b=q_b(n,\lambda,\Lambda)>n$, $\sigma\in(0,1]$,
$\F\in L^q(\Omega;\mathbb R^n)$, and
$v\in W^{1,2}(\Omega)\cap C(\overline\Omega)$ satisfy
\[
 L_uv=\diver\F\quad\text{weakly in }\Omega,
 \qquad
 v=\phi\quad\text{on }\partial\Omega,
\]
where $\phi\in C^\sigma(\partial\Omega)$.  Then there exist $\gamma\in (0, 1)$ and $C>0$, which depend only on $n, \lambda, \Lambda, q,  \sigma, \|u\|_{C^3(\partial \Omega)}$, $R$, and the $C^3$ regularity of $\partial \Omega$,
such that
\begin{equation}\label{eq:global-finite-Lq-holder}
 \|v\|_{C^\gamma(\overline\Omega)}
 \le C\left(
      \|\phi\|_{C^\sigma(\partial\Omega)}
      +\|\F\|_{L^q(\Omega)}
      \right).
\end{equation}
\end{theorem}
\begin{remark} 
\leavevmode
\begin{enumerate}
\item[(i)] (Two dimensions) 
We state our results including Theorems \ref{thm:main} and 
\ref{thm:global-finite-Lq} for dimensions $n\geq 3$. However, the methods of the proofs are applicable to the two-dimensional case. For this, we only need to change the localization of the superlevel sets of the Green's function of the 
Monge--Amp\`ere operator. In particular, when $n=2$, we need to replace  \eqref{eq:levels} by
\begin{equation*}
 \{x\in S:g_S(x,y)>k\}
 \subset S_u(y, 2\eta h 2^{-k/\tau}),
\end{equation*}
and \eqref{eq:global-level-localization-Lq} by
\begin{equation*}
 \{x\in V:g_V(x,y)>k\}
 \subset S_u(y,  h 2^{-k/{\bar C}}),
\end{equation*}
if $V$ is contained in a section of $u$ with height $h$.
\item[(ii)] (Global assumptions) In the proof of Theorem \ref{thm:global-finite-Lq}, the refined global analysis of the Green's function plays an important role. This was developed in 
\cite[Section 14.4]{Le24} under the global assumption $u\in C^{1,1}(\overline{\Omega})$ on the convex potential function.
However, our global H\"older estimates do not depend on the bounds on the Hessian of $u$. This point will be crucial in establishing the nonlinear stability result in Theorem \ref{thm:MA-stability} via approximations. 
\end{enumerate}
\end{remark}
In the next two subsections, we will describe some applications of Theorems \ref{thm:main} and 
\ref{thm:global-finite-Lq}.

\subsection{Application to dual semigeostrophic equations}\label{subsec:intro-dual-semi-eq}
Our first application is concerned with the semigeostrophic equations, which 
we give a brief introduction here.

\medskip

The semigeostrophic equations are simplified models used in meteorology to
describe large-scale atmospheric flows, particularly in mid-latitude weather
systems where the Earth's rotation plays a crucial role. As explained by
Benamou and Brenier \cite[Section~2.2]{BB}, Loeper \cite[Section~1.1]{Lo06},
and Cullen \cite{Cu06}, these equations can be formally derived from the
three-dimensional incompressible Euler equations under the Boussinesq and
hydrostatic approximations, in a regime where the Coriolis force is dominant.
For large-scale atmospheric motions, the Coriolis effect dominates the
nonlinear advection term, and much of the mathematical literature has focused
on the two-dimensional periodic setting; see, for instance,
\cite{ACDF,Le18}. The three-dimensional case is physically relevant but
involves additional difficulties; see \cite{ACDF14}. 
Our application concerns the periodic dual system on the three-dimensional torus
$\mathbb T^3:=\mathbb R^3/\mathbb Z^3$.

\medskip

The semigeostrophic equations take the form
\begin{equation}\label{eq:SG}
 \left\{
 \begin{aligned}
 \partial_t u_t^g+(u_t\cdot D)u_t^g+D p_t
     &=-Ju_t+m_t e_3,\\
 \partial_t m_t+(u_t\cdot D)m_t&=0,\\
 u_t^g&=J D p_t,\\
 \diver u_t&=0,\\
 p_{t=0}&=p^0,
 \end{aligned}
 \right.
\end{equation}
where $e_3=(0,0,1)^T$ and
\[
 J=\begin{pmatrix}
 0&-1&0\\
 1&0&0\\
 0&0&0
 \end{pmatrix}.
\]
The functions $u_t$, $p_t$, and $m_t$ denote the velocity, pressure, and
buoyancy, respectively, and $u_t^g$ is the semigeostrophic wind. We use the
notation $f_t=f(t,\cdot)$.

Substituting $u_t^g=J D p_t$ and introducing
\[
 P_t(x):=p_t(x)+\frac12(x_1^2+x_2^2),
\]
the system \eqref{eq:SG} can be rewritten as
\begin{equation}\label{eq:SG-primal}
 \left\{
 \begin{aligned}
 \partial_t D P_t+D^2P_t\,u_t
     &=J(D P_t-x),\\
 \diver u_t&=0,\\
 P_{t=0}(x)&=p^0(x)+\frac12(x_1^2+x_2^2).
 \end{aligned}
 \right.
\end{equation}
Energetic considerations make convexity of $P_t$ natural; see
\cite[Section~3.2]{Cu06}. Since the existence theory for the primal system is
difficult, one passes to the dual formulation. If $P_t^*$ denotes the
Legendre transform of $P_t$, the dual semigeostrophic equations are
\begin{equation}\label{eq:dual-SG-intro}
 \left\{
 \begin{aligned}
 \partial_t\rho_t+\operatorname{div}(\rho_t U_t)&=0,\\
 U_t(x)&=J\bigl(x-D P_t^*(x)\bigr),\\
 \det D^2P_t^*&=\rho_t,\\
 \rho_{t=0}&=\rho^0,
 \end{aligned}
 \right.
 \qquad (t,x)\in(0,\infty)\times\mathbb T^3,
\end{equation}
with $P_t^*$ convex and
\begin{equation}\label{eq:period-intro}
 P_t^*(x)-\frac12|x|^2\quad\text{is $\mathbb Z^3$-periodic.}
\end{equation}
The Monge--Amp\`ere equation in \eqref{eq:dual-SG-intro} is understood in the Aleksandrov sense for weak
solutions. This dual formulation and its relation to the primal variables are
discussed in \cite{BB,Lo05,ACDF14}. Formally, once a sufficiently regular
dual solution is available, one can recover the physical variables through
the Legendre transform. The main obstruction is therefore to obtain enough
regularity for the dual potential.

\medskip

The spatial regularity is by now well understood through the regularity theory
of the Monge--Amp\`ere equation; see, for instance, \cite{Fi, Fi18, Le24}. It is then
natural to investigate regularity in time. Differentiating
$\det D^2P_t^*=\rho_t$ with respect to $t$ and using the continuity equation
in \eqref{eq:dual-SG-intro}, we obtain
\[
 \cof(D^2P_t^*)^{ij}D_{ij}(\partial_tP_t^*)
 =\partial_t\rho_t
 =\operatorname{div}(-\rho_tU_t)
 =:\operatorname{div}F_t.
\]
Since the cofactor matrix is divergence-free, $v_t:=\partial_tP_t^*$
satisfies the linearized Monge--Amp\`ere equation
\begin{equation}\label{eq:parP-intro}
 \operatorname{div}\!\left(
   \rho_t(D^2P_t^*)^{-1} D\partial_tP_t^*
 \right)=\operatorname{div}F_t,
 \qquad F_t=-\rho_tU_t.
\end{equation}
If the initial density $\rho^0$ is bounded from below and above by positive constants, we can show that $F_t$ is uniformly bounded in
$L^\infty(\mathbb T^3)$. Moreover, the time-regularity estimates developed
in \cite{Lo05,ACDF,ACDF14} give, after fixing the additive constant of
$P_t^*$, an $L^p$ bound for $\partial_tP_t^*$ for some $p>1$. Thus the
H\"older continuity of $\partial_tP_t^*$ reduces exactly to the estimate of
Theorem~\ref{thm:main}.

\medskip

Regarding previous time-regularity results, Loeper
\cite[Theorem~9.2]{Lo05} proved that if the initial density is sufficiently
close to a constant, namely
$1-\varepsilon_0\leq\rho^0\leq1+\varepsilon_0$ on $\mathbb T^3$, then
$\partial_tP_t$ and $\partial_tP_t^*$ are 
uniformly H\"older continuous in space, with an essentially uniform bound in time. Le \cite[Theorem~1.2]{Le18} later removed the
small-oscillation assumption on the two-dimensional torus. Combining
Theorem~\ref{thm:main} with the three-dimensional estimates in
\cite{ACDF14}, and otherwise following the argument of
\cite[Section~4]{Le18} (see also \cite[Theorem 15.13]{Le24}), gives the following three-dimensional result.

\begin{theorem}[H\"older regularity of three-dimensional dual semigeostrophic equations]\label{thm:holder-dual}
Let $\rho^0$ be a probability density on $\mathbb T^3$ satisfying
\[
 0<\lambda\leq\rho^0\leq\Lambda<\infty.
\]
Let $(\rho_t,P_t^*)$ be a solution of
\eqref{eq:dual-SG-intro}--\eqref{eq:period-intro} with normalization $\int_{\mathbb T^3}P_t^*\, dx=0$, and let $P_t$ be the
Legendre transform of $P_t^*$. Then there exist
$\alpha=\alpha(\lambda,\Lambda)\in(0,1)$ and
$C=C(\lambda,\Lambda)>0$ such that
\begin{equation}\label{eq:sg-main-intro}
 \|\partial_tP_t^*\|_{L^\infty((0,\infty);C^\alpha(\mathbb T^3))}
 +\|\partial_tP_t\|_{L^\infty((0,\infty);C^\alpha(\mathbb T^3))}
 \leq C.
\end{equation}
\end{theorem}

The proof is given in Section~\ref{sec:sg}. The crucial point, compared with
the earlier argument, is that Theorem~\ref{thm:main} only requires an
$L^p$ norm of the solution for some $p>1$. Hence, the known estimate
$D\partial_tP_t^*\in L^{1+\kappa}$, together with Poincar\'e's
inequality after normalization, supplies exactly the integrability needed
for \eqref{eq:main}.

\subsection{Application to quantitative homogenization for the Monge--Amp\`ere equation}\label{subsec:intro-homogenization}
A second application concerns quantitative homogenization for the Monge--Amp\`ere equation. Here the domain and the boundary data are fixed, 
while the prescribed Monge--Amp\`ere density oscillates on a
scale tending to zero. The question is how these oscillations affect the convex solution at a fixed spatial scale. 
Unlike a perturbation with small amplitude, a rapidly oscillating density need not be close to its average in any strong $L^p$ norm. 
An effective stability estimate must therefore detect cancellation in the density, rather than only its pointwise size.

\medskip

Let $\Omega\subset\R^n$ ($n\geq2$) be a bounded smooth uniformly convex
domain, and $\mathbb T^n=\R^n/\mathbb Z^n$ so it is $[0, 1]^n$ with opposite faces identified.
Let $f\in L^\infty(\mathbb T^n)$, with periodic extension to $\R^n$, satisfy
\begin{equation}\label{eq:hom-intro-density}
	0<\lambda\leq f\leq\Lambda,
	\qquad \bar f:=\int_{[0,1]^n}f(y)\,\mathrm{d}y.
\end{equation}
Fix also a boundary datum $\varphi\in C^3(\partial\Omega)$.
We consider the convex Aleksandrov solutions of
\begin{equation}\label{eq:hom-intro-problems}
	\left\{
	\begin{aligned}
		\det D^2u_\varepsilon&=f(x/\varepsilon)
		&&\text{in }\Omega,\\
		\det D^2\bar u&=\bar f
		&&\text{in }\Omega,\\
		u_\varepsilon=\bar u&=\varphi
		&&\text{on }\partial\Omega.
	\end{aligned}
	\right.
\end{equation}
Periodic averaging (see \cite[pp. 5-6]{ZKO}) gives
$f(\cdot/\varepsilon)\rightharpoonup^*\bar f$ in $L^\infty(\Omega)$.
Together with the compactness and weak continuity of Monge--Amp\`ere
measures, this identifies $\bar u$ as the uniform limit of
$u_\varepsilon$; see \cite[Chapter 3]{Le24}.
This qualitative argument does not provide a convergence rate.

\medskip

Caffarelli and Li \cite{CL04} established a quantitative estimate in
their study of entire solutions with periodic Monge--Amp\`ere density.
For a normalized convex domain $B_1\subset O\subset B_n$ and a positive
periodic density $g$ of cell average one, let $w$ and $\bar w$ have zero
boundary values on $\p O$ and Monge--Amp\`ere densities $g$ and $1$, respectively. Their
Theorem~0.3 gives
\[
\|w-\bar w\|_{L^\infty(O)}
\leq C\sum_{i=1}^n|\ell_i|^b,
\]
where $\ell_1,\ldots,\ell_n$ are linearly independent period vectors
and $C,b>0$ depend only on $n$ and an upper bound for $g$.
Their Remark~0.7 extends this estimate to bounded periodic densities
by approximation. For periods $\varepsilon e_i$, it gives a rate
$C\varepsilon^b$. The proof uses periodic correctors and interior
derivative estimates for the constant-density solution.
Li and Lu \cite{LL22} subsequently proved the existence of periodic
correctors and extended the entire-solution classification to 
measurable Monge--Amp\`ere densities bounded away from zero. Recently, Jin, Li, Tran and Tu \cite{JLTT} classified global convex solutions to $\det D^2 u=\mu$ for all nonzero, nonnegative locally finite periodic Borel measures $\mu$ on $\R^n$. The above-mentioned results concern
the structure of entire solutions. Here, we seek a uniform rate for
the fixed-domain Dirichlet problem.

\medskip

The difficulty is not the identification of the effective density.
Indeed, for every $1\leq p<\infty$, periodic averaging gives
\[
\lim_{\varepsilon\to0^+}
\|f(\cdot/\varepsilon)-\bar f\|_{L^p(\Omega)}^p
=|\Omega|\int_{[0,1]^n}|f(y)-\bar f|^p\,\mathrm{d}y,
\]
which is positive unless $f$ is constant almost everywhere. Thus, an
estimate depending only on this strong norm cannot capture the
small period. Moreover, the Hessian determinant bounds alone do not control the
ellipticity ratio of the coefficient matrix of the linearized operator. Consequently,
classical uniformly elliptic estimates do not directly give bounds
with constants depending only on the Monge--Amp\`ere density bounds.

\medskip

We state an application to the periodic homogenization of the Monge--Amp\`ere equation in all dimensions $n\geq2$.
For $n\geq3$, the global H\"older estimate needed below is provided by
Theorem~\ref{thm:global-finite-Lq}; in dimension two, we use
\cite[Theorem~1.2]{Le20} (see also \cite[Theorem 15.9]{Le24}). Along the density homotopy, the boundary datum
is fixed, so its derivative has zero boundary trace. 
The density-homotopy argument gives the following global estimate.

\begin{theorem}[Linear rate in periodic homogenization]\label{cor:MA-homogenization}
Let  $\Omega\subset\R^n$ ($n\geq2$) be a bounded smooth uniformly
convex domain. Let $\varphi\in C^3(\partial\Omega)$, and let
$f\in L^\infty(\mathbb T^n)$ satisfy \eqref{eq:hom-intro-density}.
Let $u_\varepsilon$ and $\bar u$ solve \eqref{eq:hom-intro-problems}.
There exist $\gamma\in(0,1)$ and $C>0$, depending only on
$n,\lambda,\Lambda,\Omega$, and $\|\varphi\|_{C^3(\partial\Omega)}$,
such that, for every $\varepsilon>0$,
\begin{equation}\label{eq:hom-intro-uniform}
 \|u_\varepsilon-\bar u\|_{L^\infty(\Omega)}
 \leq C\varepsilon\|f-\bar f\|_{L^\infty(\mathbb T^n)}.
\end{equation}
Moreover,
\begin{equation}\label{eq:hom-intro-holder}
 [u_\varepsilon-\bar u]_{C^\gamma(\overline\Omega)}
 \leq C\varepsilon\|f-\bar f\|_{L^\infty(\mathbb T^n)}.
\end{equation}
\end{theorem}

The above estimates give a linear rate in the uniformly nondegenerate class of Monge--Amp\`ere equations
on a fixed smooth uniformly convex domain. These assumptions are
stronger than the domain assumptions in \cite{CL04}, whose normalized
estimate also does not require a fixed positive lower density bound; see also Remark \ref{rkJLTT}.
We do not claim that the linear rate is optimal. Nevertheless, this result is reminiscent of the recent work of Tran and Yu \cite{TY} on the periodic homogenization of convex Hamilton–Jacobi equations, where the optimal rate of convergence was shown to be linear. The global
H\"older linear rate in \eqref{eq:hom-intro-holder} follows from the global
nonlinear stability in Theorem~\ref{thm:MA-stability}; in particular,
it is not obtained by interpolating \eqref{eq:hom-intro-uniform} with a
bound independent of $\varepsilon$, which would generally lose a power of
$\varepsilon$.

\medskip

In fact, the periodic homogenization follows from a global nonlinear stability
estimate with a negative norm for the Monge--Amp\`ere equation in Theorem \ref{thm:MA-stability} below that does not require periodicity. 

\medskip

For $f\in L^\infty(\Omega)$, set
\begin{equation}
\label{hdiv}
 \|f\|_{\mathrm{div},\infty;\Omega}
 :=\inf\left\{\|\F\|_{L^\infty(\Omega)}:
       \F\in L^\infty(\Omega;\R^n),\quad
       \diver\F=f\ \text{in }\mathcal D'(\Omega)\right\}.
\end{equation}
In the set of competitors, no boundary condition is imposed on $\F$. The infimum is finite:
extend $f$ by zero to a box containing $\Omega$ and integrate in one
coordinate to obtain a bounded vector field with divergence $f$. This simple argument also shows that
\[ \|f\|_{\mathrm{div},\infty;\Omega} \leq C(\Omega) \|f\|_{L^\infty(\Omega)}.\]

\begin{theorem}[Nonlinear stability
for the Monge--Amp\`ere equation]
\label{thm:MA-stability}
 Let $\Omega\subset\R^n$ $(n\geq 2)$ be a 
smooth uniformly convex domain.  Let
$0<\lambda\leq\Lambda<\infty$ and $\varphi\in C^3(\partial\Omega)$. Let  $f_0,f_1\in L^\infty(\Omega)$ satisfy
$\lambda\leq f_i\leq\Lambda$, and let 
$u_i\in C(\overline{\Omega})$ ($i=0, 1$)
be the convex Aleksandrov solution of
\[
 \det D^2u_i=f_i\quad\text{in }\Omega,
 \qquad u_i=\varphi\quad\text{on }\partial\Omega.
\]
Then, there exist $\gamma\in(0,1)$ and $C>0$, depending only on
$n,\lambda,\Lambda,\Omega$, and $\|\varphi\|_{C^3(\partial\Omega)}$,
such that
\begin{equation}\label{eq:MA-stability}
 \|u_1-u_0\|_{C^\gamma(\overline\Omega)}
 \leq C\|f_1-f_0\|_{\mathrm{div},\infty;\Omega}.
\end{equation}
\end{theorem}

Theorem~\ref{thm:MA-stability} shows that the solution map $f\mapsto u$ is Lipschitz, with values in
$C^\gamma(\overline\Omega)$, with respect to a norm that measures the
smallness of a bounded divergence primitive rather than the amplitude of
the density difference. Equality of the total masses of $f_0$ and
$f_1$ is not required.

\medskip

For the periodic problem, the relevant smallness is elementary.
The Poisson equation
$\Delta_y\psi=f-\bar f$ on $\mathbb T^n$ has a zero-mean solution for which
$G=D_y\psi$ satisfies
\[
\diver_yG=f-\bar f,
\qquad
\|G\|_{L^\infty(\mathbb T^n)}
\leq C(n)\|f-\bar f\|_{L^\infty(\mathbb T^n)}.
\]
The bound follows from the periodic $W^{2,p}$ estimate with $p>n$
and Sobolev embedding; see \cite[Chapters~7 and~9]{GT01}.
Hence $\F_\varepsilon(x):=\varepsilon G(x/\varepsilon)$ satisfies
\[
f(x/\varepsilon)-\bar f=\diver\F_\varepsilon,
\qquad
\|\F_\varepsilon\|_{L^\infty(\Omega)}
\leq C(n)\varepsilon\|f-\bar f\|_{L^\infty(\mathbb T^n)}.
\]
This is where the small period enters the estimate.

\medskip

To explain the nonlinear step, assume first that the densities $f_0$ and $f_1$ are
smooth. Observe that $w:=u_1-u_0$ satisfies the linear elliptic equation
\begin{equation}
\label{u01w}
\Big(\int_0^1
\cof\big((1-s)D^2u_0+sD^2u_1\big)\,\mathrm{d}s\Big)^{ij}D_{ij}w
=f_1-f_0.
\end{equation}
This averaged coefficient matrix need not be the cofactor Hessian of
a convex potential satisfying the same Hessian determinant bounds as those of $u_0$ and $u_1$. 

\medskip
Instead, we
interpolate the densities and solve
\[
\det D^2u_s=(1-s)f_0+sf_1\quad\text{in }\Omega,
\qquad u_s=\varphi\quad\text{on }\partial\Omega,
\qquad 0\leq s\leq1.
\]
The global $C^{2,\alpha}(\overline{\Omega})$ solvability  and the Implicit Function Theorem
justify differentiation in $s$; see \cite{GT01, TW08}. 
If $\diver\F=f_1-f_0$, then $v_s:=\partial_su_s$ solves
\[
L_{u_s}v_s=\diver\F\quad\text{in }\Omega,
\qquad v_s=0\quad\text{on }\partial\Omega.
\]

The Hessian determinant bounds are preserved along this path.
For $n\geq3$, take $q=q_b+1$ in
Theorem~\ref{thm:global-finite-Lq}. Since $\F\in L^\infty(\Omega)$,
that theorem gives
\[
 \|v_s\|_{C^\gamma(\overline\Omega)}
 \leq C\|\F\|_{L^q(\Omega)}
 \leq C\|\F\|_{L^\infty(\Omega)},
\]
uniformly in $s$. In dimension two, the same global estimate follows
from \cite[Theorem~1.2]{Le20}. Integrating
$u_1-u_0=\int_0^1v_s\,\mathrm{d}s$ and minimizing over $\F$ give the global stability estimate.

\medskip

Only the resulting nonlinear estimate is passed to bounded
measurable densities, using smooth approximations that preserve the
divergence identity. The proof, including this approximation and
Theorem~\ref{cor:MA-homogenization}, is given in
Section~\ref{sec:homogenization}.
\begin{remark} We make several remarks on Theorem \ref{thm:MA-stability}.
\begin{enumerate}
\item[(i)](Negative norm) We emphasize that the main point in the nonlinear stability estimates in Theorem \ref{thm:MA-stability} is the {\it negative norm} $\|f_1-f_0\|_{\mathrm{div},\infty;\Omega}$ on the right-hand side. If this is replaced by the larger quantity $\|f_1-f_0\|_{L^\infty(\Omega)}$, then the resulting estimate follows from the homotopy argument above and  replacing Theorem~\ref{thm:global-finite-Lq} by the global H\"older estimates for the linearized Monge--Amp\`ere equation with $L^n$ right-hand side; see \cite[Theorem 13.2]{Le24}.  
\item[(ii)] ($C^{1,\alpha}$ stability with a small power rate) Under the hypotheses of Theorem \ref{thm:MA-stability}, the functions $u_0, u_1$ are $C^{1,\beta}(\overline{\Omega})$ for some $\beta=\beta(n,\lambda,\Lambda)\in (0, 1)$; see \cite[Theorem 9.5]{Le24}. 
Let $\alpha\in (0,\beta)$. Then,
using the interpolation inequality
\[\|u_1-u_0\|_{C^{1,\alpha}(\overline\Omega)}
\le C(n,\Omega,\alpha,\beta)
\|u_1-u_0\|_{L^\infty(\Omega)}^{\frac{\beta-\alpha}{1+\beta}}
\|u_1-u_0\|_{C^{1,\beta}(\overline\Omega)}^{\frac{1+\alpha}{1+\beta}},\]
and the discussion in (i), we obtain the following $C^{1,\alpha}$ stability with a small power rate:
 \[\|u_1-u_0\|_{C^{1,\alpha}(\overline\Omega)}
\le C(n,\Omega,\alpha,\beta,\lambda,\Lambda, \|\varphi\|_{C^3(\partial\Omega)})
\|f_1-f_0\|_{L^\infty(\Omega)}^{\frac{\beta-\alpha}{1+\beta}}.\]
\item[(iii)] ($C^{2,\beta}$ stability with a linear rate) If we impose additionally $f_0, f_1\in C^{\beta} (\overline{\Omega})$, then, in view of \cite[Theorem 1.1]{TW08}, $u_0, u_1\in C^{2,\beta}(\overline{\Omega})$ . Therefore, the averaged coefficient matrix in \eqref{u01w} is uniformly elliptic and belongs to 
$C^{\beta} (\overline{\Omega})$. Consequently, by the classical Schauder theory, instead of \eqref{eq:MA-stability}, we have a stronger estimate
\begin{equation}\label{stability2}
 \|u_1-u_0\|_{C^{2, \beta}(\overline\Omega)}
 \leq C\|f_1-f_0\|_{C^{\beta} (\overline{\Omega})}.
\end{equation}
\item[(iv)] (Sobolev space stability) For the stability in Sobolev space $W^{2,\gamma}$ where $\gamma>1$, see \cite[Theorem 1.1]{DPF} and \cite[Proposition 3.18]{LN}. 
\end{enumerate}
\end{remark}
\begin{remark}
\label{rkJLTT}
In the preprint \cite{JLTT2} posted right after our paper, Jin, Li, Tran, and Tu  independently obtained the linear convergence rate \eqref{eq:hom-intro-uniform} in Theorem~\ref{cor:MA-homogenization} when the effective solution $\bar u\in C^{2,1}(\overline{\Omega})$; see \cite[Corollary 1.6]{JLTT2}. Their method is quite different from ours and it also gives a higher exponent for $\e$ in the convergence rates when $\bar u$ has better regularity, among other hypotheses.
\end{remark}
%%%%%%%%%%%%%%%%%%%%%%%%%%%%%%%%%%%%%%

\medskip
\paragraph{\bf AI assistance.} The authors have worked on linearized Monge--Amp\`ere equations in divergence form during the last decade \cite{CWZ25, GL26, KLWZ, Le18, Le20, Le24}. The unconditional H\"older estimates previously obtained in this line of work were restricted to dimension two. A main obstacle to extending them to dimensions $n\geq3$ was the lack of a suitable integrability estimate for the gradient of the Green's function of the linearized Monge--Amp\`ere operator.
The argument leading to Proposition~\ref{prop:green}, which provides the missing $L^1$ estimate for the gradient of the Green's function, was first produced in repeated interactions with OpenAI's GPT-6 Astra on September 16, 2026. The authors subsequently checked the argument of Proposition~\ref{prop:green} in detail, verified the proof independently, and take full responsibility for its correctness. This estimate supplied the key missing ingredient for completing the higher-dimensional argument. With it, the authors could complete their research program without further assistance.
 
%%%%%%%%%%%%%%%%%%%%%%%%%%%%%%%%%%%%%%

\medskip

The rest of the paper is organized as follows. 
We recall properties of sections of solutions to the Monge--Amp\`ere equation and the Green's functions of the linearized Monge--Amp\`{e}re operator in Section \ref{sec_prelim}.
Section~\ref{sec:green} is devoted to the proof of Proposition~\ref{prop:green} and the 
uniform estimate for zero-boundary solutions. The proof of Theorem \ref{thm:main} is completed in Section~\ref{sec:holder} by a splitting argument. We then apply the resulting estimate to the three-dimensional periodic dual semigeostrophic equations in Section~\ref{sec:sg}. The global higher
integrability estimate for the gradient of the Green's function and Theorem \ref{thm:global-finite-Lq} are proved in Section~\ref{sec:global}, while Section~\ref{sec:homogenization} concludes with the negative-norm stability estimate and the periodic homogenization rate.

\section{Preliminaries} 
\label{sec_prelim}

In this section, we recall some definitions and  properties of sections of solutions to the Monge--Amp\`ere equation and the Green's functions of the linearized Monge--Amp\`{e}re operator.

We recall the notion of sections of a convex function.
\begin{defn}[Section] Let $\Omega$ be a bounded convex domain in $\R^n$ $(n\geq 2)$.
	Let $u\in C^1(\Omega)\cap C(\overline{\Omega})$ be a convex function and $h > 0$. 
If $x_0 \in \overline \Omega$, then the \textit{section of $u$ centered at $x_0$ with height $h$} is defined by 
	           \[
	               S_u(x_0, h):= \big\{x \in \overline{\Omega}: u(x) < u(x_0) + Du(x_0)\cdot (x - x_0) + h\big\}.
	           \]
	           In the case of $x_0\in\p\Omega$, we require that $u$ is differentiable at $x_0$.
\end{defn}
Following \cite[Definition 1.3]{GL26}, we define the notion of a proper set as follows.
\begin{defn}[Proper set]
A nonempty open set $V\subset\R^n$  is called {\it proper} if it satisfies an exterior cone condition at every boundary point. Examples include intersections of a bounded convex domain $\Omega\subset\R^n$ with sections of a convex function $u\in C^1(\overline{\Omega})$ or open balls.
\end{defn}
The Green's function of the linearized Monge--Amp\`{e}re operator is defined as follows; see \cite[Chapter 14]{Le24}.
\begin{defn}[Green's function of the linearized Monge--Amp\`{e}re operator] 
\label{Gdefn}
Let $\Omega$ be a bounded convex domain in $\R^n$ and $u\in C^3(\Omega)$ be a convex function satisfying \eqref{eq:det}. Let $V \subset  \Omega$ be open. 
  Let $\delta_y$ be the Dirac measure giving the unit mass to $y$. 
Then, for each $y \in V$, there exists a unique function $g_V(\cdot, y): V \rightarrow[0, \infty]$ with the following properties:
    \begin{enumerate}
        \item[(i)] $g_V(\cdot, y) \in W_0^{1, q}(V) \cap W^{1,2}\left(V \backslash B_r(y)\right)$ for all $q<\frac{n}{n-1}$ and all $r>0$.
        \item[(ii)] $g_V(\cdot, y)$ is a weak solution of
            \[
                    -D_i\left(U^{i j} D_j g_V(\cdot, y)\right)  =\delta_y \quad \text { in } V, \quad
	               g_V(\cdot, y)  =0 \quad \text { on } \p V,
            \]
            that is, denoting $D_jg_V(x, y) = \frac{\p}{\p x_j}g_V(x, y)$, we have
	           \begin{equation}\label{Gdef}
	            \int_V U^{i j} D_j g_V(x, y) D_i \psi(x) \,\mathrm{d}x=\psi(y) \quad \text{ for all } \psi \in C_c^{\infty}(V).
	           \end{equation}
             \end{enumerate}
  
    We call $g_V(\cdot, y)$ the \textit{Green's function} of the linearized Monge--Amp\`{e}re operator $L_u=D_i\left(U^{i j} D_j\right)$ in $V$ with \textit{pole} $y$, and  set $g_V(y, y) = +\infty$.
\end{defn}

\begin{remark}[Representation formula using the Green's function]
\label{RepRk}
 Assume that $V\subset\R^n$ is proper. 
		If $\varphi \in L^n(V)$, then there exists a unique solution $\psi \in W_{\text{loc}}^{2, n}(V) \cap W_0^{1,2}(V) \cap C(\overline{V})$ to
		 	\[
		 		-U^{ij} D_{ij} \psi = \varphi \quad\text{ in } V \quad \text{ and } \psi = 0 \quad\text{ on } \partial V. 
		 	\]
			Use $\psi$ as a test function in \eqref{Gdef}, then (see also \cite[Equation (14.5)]{Le24}):
			\begin{equation}\label{Grep}
				\psi(y)=\int_V g_V(\cdot, y) \varphi \,\mathrm{d}x.
			\end{equation}	
\end{remark}

For interior estimates, we use Definition \ref{Gdefn} and Remark \ref{RepRk} for \(V\Subset\Omega\), while for boundary estimates, we additionally
impose the global assumptions of Section \ref{sec:global}, including \(u\in C^{1,1}(\overline\Omega)\).

\begin{remark} (Volume and interior $C^{1,\alpha}$ estimates)
Let $u \in C(\Omega)$ be a strictly convex function satisfying \eqref{eq:det} in the sense of Aleksandrov, where $\Omega\subset \R^n$ $(n\geq 3)$. Then  $u\in C^{1,\alpha}(\Omega)$ where 
 $\alpha=\alpha(n,\lambda,\Lambda)\in(0,1)$; see \cite[Theorem 5.21]{Le24}.
  For later reference, we set
\begin{equation}\label{eq:exponents}
 a=\frac{n-2}{2},\qquad \beta=\frac{\alpha}{1+\alpha}.
\end{equation}
We use the following volume estimates (see \cite[Lemma 5.6 (i)]{Le24})
\begin{equation}
\label{eq:volume}
 C^{-1} s^{n/2}\leq |S_u(z,s)|\leq C s^{n/2}
       \qquad\text{if }S_u(z,s)\Subset\Omega,
       \end{equation}
   where $C=C(n,\lambda,\Lambda)$, and consequences of \cite[Lemma~5.6(ii) and
Corollary~5.22]{Le24}:
    \begin{equation}   
 [Du]_{C^\alpha(\overline{S_{2h_0}})}\leq M,
       \quad\operatorname{dist}(\overline S_{3h_0/2},\partial S_{2h_0})\geq d_0>0  \quad\text{if }S_{4h_0}=S_u(x_0, 4h_0)\Subset\Omega.
       \label{eq:geometry}
\end{equation}
Here $M$ and $d_0$ depend only on $n,\lambda,\Lambda$, $h_0$ and $\diam(S_{4h_0})$.
Specifically, in \cite{Le24}, Corollary~5.22 can be applied at height $3h_0$
with $\tau=4/3$ and $\delta=2/3$; Lemma~5.6(ii), at height $2h_0$ and
ratio $3/4$, gives $d_0\geq C^{-1} h_0^{n/2}[\diam(S_{4h_0})]^{1-n}$. 
\end{remark}

\medskip

We recall a localization property of superlevel sets of the Green's function of the linearized Monge--Amp\`ere operator  that will be needed in the proof of Proposition~\ref{prop:green}. 
\begin{lemma} Let $u \in C^3(\Omega)$ be a convex function satisfying \eqref{eq:det}, where $\Omega\subset \R^n$ $(n\geq 3)$.
Assume $S_u(z,2h)\Subset \Omega$ where $z\in\Omega$ and $h>0$. Let $g_{S}(\cdot, y)$ be the Green's function of the linearized Monge--Amp\`{e}re operator $D_i(U^{ij} D_j)$ in $S=S_u(z, h)$ with pole $y\in S$. 
\begin{enumerate}
\item[(i)] (Localization of superlevel sets) There are constants
$c_l=c_l(n,\lambda,\Lambda)>0$ and $\tau=\tau(n,\lambda,\Lambda)>1$ such that, for every $y\in S$ and $k\geq\tau h^{-a}$,
\begin{equation}\label{eq:levels}
 \big\{x\in S:g_S(x,y)>k\big\}
 \subset S_u(y,c_l k^{-1/a})\subset S_u(y,2c_l k^{-1/a})\Subset S_u(z,3h/2).
\end{equation}
\item[(ii)] (Truncation identity) Let $y\in S$. Then, for any $k>0$, we have
the truncation identity (see
\cite[(14.23), equivalently (14.71)]{Le24})
\begin{equation}\label{eq:energy}
 \int_{\big\{g_S(\cdot,y)\leq k\big\}}U^{ij}D_i g_S(\cdot,y)D_j g_S(\cdot,y)\dd=k.
\end{equation}
\end{enumerate}
\end{lemma}
\begin{proof} We indicate the derivation of \eqref{eq:levels} for the reader's convenience. 
Indeed, \cite[Theorem~5.13(iii)]{Le24}, applied in $S_u(z,2h)$, gives a
constant $\eta=\eta (n,\lambda,\Lambda)>0$ such that $S_u(y,\eta h)\Subset S_u(z,3h/2)$ for all
$y\in S$. 
 If $k\geq\tau h^{-a}$ where $\tau=\tau(n,\lambda,\Lambda)>1$ is large, then \cite[Lemma~4.1(i)]{GL26} gives the first
inclusion in \eqref{eq:levels} with a constant
$c_l=c_l(n,\lambda,\Lambda)>0$. Increasing $\tau$ ensures that
$2c_l k^{-1/a}<\eta h/2$, and this gives the last inclusion, as desired.
\end{proof}

\begin{remark}[Euclidean gradient versus Monge--Amp\`ere gradient] Let $u \in C^3(\Omega)$ be a convex function satisfying $\det D^2 u\geq\lambda>0$.
By \cite[Lemma~2.58(iii)]{Le24}, we have for all functions $v$,
\begin{equation}
\label{DgDelta}
|Dv|^2\leq \lambda^{-1}(\Delta u) U^{ij}D_i vD_j v\equiv \lambda^{-1}(\Delta u) (\det D^2 u) |(D^2 u)^{-1/2} Dv|^2.
\end{equation}
\end{remark}
\section{The Green's function gradient estimate}\label{sec:green}
In this section, we prove Proposition~\ref{prop:green} and its extension to higher integrability. Unless otherwise indicated, we assume $n\geq3$, \eqref{eq:det} and \eqref{eq:outer}, so
\[0<\lambda\leq\det D^2u\leq\Lambda\qquad\text{in }\Omega\subset\R^n,\quad S_{4h_0}=S_u(x_0, 4h_0)\Subset\Omega.\]

We let $g_S(\cdot, y)$ be the Green's function of the linearized Monge--Amp\`{e}re operator $L_u=D_i(U^{i j} D_j)$ in $S\subset\Omega$ with pole $y\in S$.

\subsection{Integrability of the Green gradient }
The following elementary consequence of \eqref{eq:geometry} provides the
bound of the mass of the Hessian used in the Green's function estimate.

\begin{lemma}[Hessian mass on a section]
\label{lem:mass}
Assume $n\geq3$, \eqref{eq:det} and \eqref{eq:outer}. 
For every $S_u(z,s)\Subset S_{3h_0/2}$, we have
\begin{equation}\label{eq:mass}
 \int_{S_u(z,s)}\Delta u\dd\leq C s^{a+2\beta},
\end{equation}
where $a$ and $\beta$ are defined as in \eqref{eq:exponents}, and $C=C(n,\lambda,\Lambda, h_0, \diam(S_{4h_0}))$.
\end{lemma}

\begin{proof} Let $S:=S_u(z,s)$. By subtracting $u(z)+Du(z)\cdot(x-z)$ from $u$, the claimed estimate remains the same, but we can now assume $u\geq0$, $u(z)=0$, and $u=s$ on $\partial S$.

For $x\in\partial S$, set $e=Du(x)/|Du(x)|$ and
$r=c s^{1/(1+\alpha)}$. The gradient does not vanish on this positive
level set. From $S_u(z,s)\Subset S_{3h_0/2}$ and \eqref{eq:volume}, we have $s\leq C_1 h_0$ where $C_1=C_1(n,\lambda,\Lambda)>0$. Thus, there exists \[c=c(n,\lambda,\Lambda, h_0,  \diam(S_{4h_0}))>0\] such that
$r\leq d_0/2$, so the segment $\{x-te: 0\leq t\leq r\}\subset S_{2h_0}$. By \eqref{eq:geometry},
\[
 0\leq u(x-re)\leq s-r|Du(x)|+\frac{M}{1+\alpha}r^{1+\alpha}.
\]
Consequently, we can find $C_2= C_2(n,\lambda,\Lambda, h_0, \diam(S_{4h_0}))$ such that
\begin{equation}\label{eq:m}
 m_s:=\sup_{\partial S}|Du|\leq C_2 s^\beta.
\end{equation}
On $\partial S$, the outward normal is $\nu=Du/|Du|$, and convexity gives
$Du(x)\cdot(x-z)\geq s$. The divergence theorem therefore yields
\[
 s\int_S\Delta u\dd
 =s\int_{\partial S}|Du|\dH
 \leq\int_{\partial S}|Du|^2(x-z)\cdot\nu\dH
 \leq n|S|m_s^2.
\]
Combining this with \eqref{eq:volume} and \eqref{eq:m} proves
\eqref{eq:mass}, since $n/2-1+2\beta=a+2\beta$.
\end{proof}

\medskip

We now establish the $L^1$ estimate for the gradient of the Green's function.
\begin{proof}[Proof of Proposition~\ref{prop:green}]
In this proof, the constants $C$ depend only on $n,\lambda,\Lambda$, $h_0$, and $\diam(S_{4h_0})$.
Fix $y\in S$.  Then, from \eqref{DgDelta}, for every measurable
$E\subset S$,
\begin{equation}\label{eq:pair}
 \int_E|Dg|\dd\leq\lambda^{-1/2}
 \left(\int_E U^{ij}D_i gD_j g\dd\right)^{1/2}
 \left(\int_E\Delta u\dd\right)^{1/2}.
\end{equation}
Let $\tau$ be as in \eqref{eq:levels} and $T=\tau h^{-a}$. Equations
\eqref{eq:energy} and \eqref{eq:mass} imply
\begin{equation}\label{eq:low}
 \int_{\{g\leq T\}}|Dg|\dd
 \leq CT^{1/2}h^{(a+2\beta)/2}\leq Ch^\beta.
\end{equation}
For $j\geq0$, let $E_j=\{2^jT<g\leq2^{j+1}T\}$. Then by \eqref{eq:mass}, \eqref{eq:levels} and \eqref{eq:energy}, we have
\[
 \int_{E_j}U^{ik}D_i gD_k g\dd\leq2^{j+1}T,
 \qquad
 \int_{E_j}\Delta u\dd
 \leq\int_{S_u(y,c_l(2^jT)^{-1/a})}\Delta u\dd
 \leq C(2^jT)^{-(a+2\beta)/a}.
\]
The last two inequalities also use $\Delta u\geq0$. By \eqref{eq:pair},
\[
 \int_{E_j}|Dg|\dd\leq C(2^jT)^{-\beta/a}
 \leq Ch^\beta 2^{-j\beta/a}.
\]
Since $a,\beta>0$, summing over $j$ and adding \eqref{eq:low} proves
\eqref{eq:green}, with a constant independent of $y$. The proposition is proved.
\end{proof}

We show that solutions to the linearized Monge--Amp\`ere equation with right-hand side being the divergence of a bounded vector field and vanishing on the boundary of a section decay to zero when the section shrinks to a point.
\begin{corollary}\label{cor:zero}
Assume $n\geq3$, \eqref{eq:det} and \eqref{eq:outer}. 
Let $S=S_u(z,h)$ with $S_u(z,2h)\Subset S_{3h_0/2}$, and let
$\F\in L^\infty(S;\R^n)$. The weak solution $w\in W_0^{1,2}(S)$ of
\[
 L_uw=\diver\F\qquad\text{in }S,\quad w=0\qquad\text{on }\p S
\]
satisfies
\begin{equation}\label{eq:zero}
 \norm{w}_{L^\infty(S)}\leq Ch^\beta\norm{\F}_{L^\infty(S)},
 \qquad C=C(n,\lambda,\Lambda, h_0, \diam(S_{4h_0})).
\end{equation}
\end{corollary}

\begin{proof}
First, consider the case where $\F$ is a smooth vector field. 
Then, the representation formula \eqref{Grep} can be applied to $w$, 
and integration by parts gives
\[
 w(y)=-\int_S g_S(x,y)\diver\F(x)\dd
      =\int_S D_xg_S(x,y)\cdot\F(x)\dd.\]
Thus, \eqref{eq:zero} follows from Proposition~\ref{prop:green}. 

\medskip

For a general bounded vector field
$\F$, we extend it by zero and mollify to obtain smooth vector fields $\F_k\to\F$ in $L^2(S)$
with $\norm{\F_k}_{L^\infty(S)}\leq\norm{\F}_{L^\infty(S)}$. The corresponding weak
solutions $w_k \in W^{1, 2}_0(S)$ converge to $w$ in $W_0^{1,2}(S)$ by the energy estimate for their
difference. Here, the existence and convergence use the ellipticity of the fixed
smooth matrix $U$ on $\overline S$. Moreover, there is a subsequence $\{w_{k_l}\}$ that converges almost everywhere to $w$. Thus, letting $l\to\infty$ in $\norm{w_{k_l}}_{L^\infty(S)}\leq Ch^\beta\norm{\F_{k_l}}_{L^\infty(S)}$ proves
\eqref{eq:zero}.
\end{proof}

\subsection{Higher integrability of the gradient of the Green's function}\label{sec:green-higher}

The proof of Proposition~\ref{prop:green} deliberately avoids any higher
integrability input for $D^2u$.  If one now invokes the classical interior
$W^{2,1+\varepsilon}$ estimate for the Monge--Amp\`ere equation of De Philippis--Figalli--Savin \cite{DFS13} and,
independently, Schmidt \cite{Sch13}, the same
level-band localization yields a genuine improvement of the integrability of 
the gradient of the Green's function in every dimension $n\geq3$.  
The result of this section is not used in the
proof of Theorem~\ref{thm:main}.

\medskip
We first record a consequence of the interior
$W^{2,1+\varepsilon}$ estimate that will be used below.

\begin{lemma}\label{lem:hessian-higher-mass}
Assume $n\geq3$, \eqref{eq:det} and \eqref{eq:outer}.
There exists $\varepsilon_0=\varepsilon_0(n,\lambda,\Lambda)>0$ such that, for every
$\delta\in(0,\varepsilon_0]$, the following holds.  If $S_u(\xi,2s)\Subset S_{3h_0/2}$, 
then
\begin{equation}\label{eq:hessian-higher-mass}
 \int_{S_u(\xi,s)} (\Delta u)^{1+\delta}\dd
 \leq C_\delta
 s^{a+2\beta-\delta(1-2\beta)},
 \qquad C_\delta=C_\delta (n,\lambda,\Lambda, h_0, \diam(S_{4h_0})).
\end{equation}
\end{lemma}

\begin{proof} From the interior $W^{2,1+\varepsilon_0}$ estimate  for the Monge--Amp\`ere equation
\cite{DFS13,Sch13}, we have (see also \cite[Equation (1.2) and Lemma 3.1]{Mal14})
for some $C_0=C_0(n,\lambda,\Lambda)$,
\begin{equation}\begin{split}\int_{S_u(\xi,s)} (\Delta u)^{1+\varepsilon_0}\,\mathrm{d}x &\leq C_0 |S_u(\xi,s)|^{-\varepsilon_0} \Big(\int_{S_u(\xi,s)} \Delta u\, \mathrm{d}x\Big)^{1+\varepsilon_0}\\ &\leq C s^{(1+\varepsilon_0) (a+ 2\beta)-\frac{n}{2} \varepsilon_0}\\&= Cs^{a+2\beta-\varepsilon_0(1-2\beta)},\quad C=C(n,\lambda,\Lambda, h_0, \diam(S_{4h_0})),\end{split}\end{equation}
where we used Lemma \ref{lem:mass} and the volume estimates \eqref{eq:volume} for sections. Note that 
\[a+2\beta -\frac{n}{2}=2\beta-1.\]
The claimed estimate \eqref{eq:hessian-higher-mass} holds for $\delta=\varepsilon_0$. For $0<\delta< \varepsilon_0$, an application of the H\"older inequality and the  volume estimates for sections gives the desired result.
\end{proof}

\begin{proposition}[Higher integrability
of the gradient of the Green's function]
\label{prop:green-higher}
Let $n\geq3$, and assume \eqref{eq:det} and \eqref{eq:outer}.  Let
$S=S_u(z,h)$ satisfy $S_u(z,2h)\Subset S_{3h_0/2}$.  Then there exists
$\varepsilon=\varepsilon(n,\lambda,\Lambda)>0$
such that, with
\begin{equation}\label{eq:green-higher-exponent}
 \vartheta=(1+\varepsilon)\beta-\frac n2\varepsilon>0,
\end{equation}
one has
\begin{equation}\label{eq:green-higher}
 \sup_{y\in S}
 \int_S |D_xg_S(x,y)|^{1+\varepsilon}\dd
 \leq C h^\vartheta,
 \qquad C=C(n,\lambda,\Lambda, h_0, \diam(S_{4h_0})).
\end{equation}
In particular, $g_S(\cdot,y)\in W^{1,1+\varepsilon}_0(S)$, uniformly in
the pole.
\end{proposition}

\begin{proof} In this proof, the constants $C$ depend only on $n,\lambda,\Lambda$, $h_0$, and $\diam(S_{4h_0})$.
Let $\varepsilon_0$ be given by Lemma~\ref{lem:hessian-higher-mass}.  Choose
\begin{equation}\label{eq:epsilon-choice-green-higher}
 0<\varepsilon=\e(n,\lambda,\Lambda)<
 \min\left\{
 \frac{\varepsilon_0}{2+\varepsilon_0},
 \frac{2\beta}{n-2\beta}
 \right\},
\end{equation}
and set
\begin{equation}\label{eq:delta-from-epsilon}
 \delta=\frac{2\varepsilon}{1-\varepsilon}.
\end{equation}
Then $0<\delta<\varepsilon_0$,
\begin{equation}\label{eq:holder-exponents-green-higher}
 1+\delta=\frac{1+\varepsilon}{1-\varepsilon},
\end{equation}
and \eqref{eq:green-higher-exponent} follows from the second inequality
in \eqref{eq:epsilon-choice-green-higher}.

\medskip

Fix $y\in S$ and write $g:=g_S(\cdot,y)$.  Then,
H\"older's inequality with exponents $2/(1+\varepsilon)$ and
$2/(1-\varepsilon)$ applied to \eqref{DgDelta} gives, for every measurable $E\subset S$,
\begin{equation}\label{eq:pair-eps-green-higher}
 \int_E|Dg|^{1+\varepsilon}\dd
 \leq C_1(n,\lambda,\Lambda)
 \left(\int_E U^{ij}D_i gD_j g\dd\right)^{\frac{1+\varepsilon}{2}}
 \left(\int_E(\Delta u)^{1+\delta}\dd\right)^{\frac{1-\varepsilon}{2}}.
\end{equation}
Let
$T=\tau h^{-a}$.  We recall \eqref{eq:levels}, in particular, 
\begin{equation}\label{eq:doubled-localization-green-higher}
 S_u\bigl(y,2c_l k^{-1/a}\bigr)
 \Subset S_u(z,3h/2)\quad
 \text{for all }\ k\geq T.
\end{equation}

For the sublevel set $\{g\leq T\}$, \eqref{eq:energy},
Lemma~\ref{lem:hessian-higher-mass}, and
\eqref{eq:pair-eps-green-higher} yield
\begin{align*}
 \int_{\{g\leq T\}}|Dg|^{1+\varepsilon}\dd
 \leq C_1 T^{\frac{1+\varepsilon}{2}}
 \left(\int_S(\Delta u)^{1+\delta}\dd\right)^{\frac{1-\varepsilon}{2}}\leq C h^\vartheta.
\end{align*}
Here and below we use $a+1=n/2$ and
\eqref{eq:delta-from-epsilon}; explicitly,
\begin{equation}\label{eq:scaling-identity-green-higher}
 -\frac{a(1+\varepsilon)}2
 +\frac{1-\varepsilon}{2}
 \bigl[a+2\beta-\delta(1-2\beta)\bigr]
 =\vartheta.
\end{equation}

For $j\geq0$, set
\[
 K_j=2^jT,
 \qquad
 E_j=\{K_j<g\leq2K_j\},
 \qquad
 s_j=c_lK_j^{-1/a}.
\]
By \eqref{eq:energy},
\[
 \int_{E_j}U^{ik}D_i gD_k g\dd\leq2K_j.
\]
Moreover, \eqref{eq:levels} and
\eqref{eq:doubled-localization-green-higher} give
\[
 E_j\subset S_u(y,s_j),
 \qquad
 S_u(y,2s_j)\Subset S_{3h_0/2}.
\]
Applying Lemma~\ref{lem:hessian-higher-mass} and
\eqref{eq:pair-eps-green-higher}, and then using
$K_j\approx s_j^{-a}$ (within a positive factor depending only on $n,\lambda$, and $\Lambda$) together with
\eqref{eq:scaling-identity-green-higher}, we obtain
\[
 \int_{E_j}|Dg|^{1+\varepsilon}\dd
 \leq C s_j^\vartheta=C h^\vartheta 2^{-j\vartheta/a}.
\]
\iffalse
since $s_j=c_l\tau^{-1/a}h\,2^{-j/a}$.
it follows that
\[
 \int_{E_j}|Dg|^{1+\varepsilon}\dd
 \leq C h^\vartheta 2^{-j\vartheta/a}.
\]
\fi
Recalling $\vartheta>0$, we find
\begin{equation*}
\begin{split}\int_S |Dg|^{1+\e}\dd &=\int_{\{g\leq T\}}|Dg|^{1+\varepsilon}\dd + \sum_{j=0}^\infty\int_{E_j}|Dg|^{1+\varepsilon}\dd \\&\leq Ch^\vartheta (1+\sum_{j=0}^\infty 2^{-j\vartheta/a}) \leq Ch^\vartheta. 
\end{split}
\end{equation*}
Thus \eqref{eq:green-higher} is established, with a constant independent of
$y$.

Finally, combining \eqref{eq:green-higher} with the $L^q$ estimate for the Green's function in \cite[Theorem 14.11]{Le24}  yields
$g\in W^{1,1+\varepsilon}_0(S)$. The proposition is proved.
\end{proof}

\begin{remark} For the quadratic potential $u(x)=|x|^2/2$, $L_u$ is the Laplace operator and the exponent in $C^{1,\alpha}$ estimate is $\alpha=1$, so $\beta=1/2$. In this case, we can take in Proposition \ref{prop:green-higher}
any $\e\in (0, 1/(n-1))$. This is consistent with the classical integrability result for the gradient of the Green's function of the Laplace operator. 
\end{remark}

\section{Interior H\"older estimate}\label{sec:holder}

We prove interior H\"older estimates for the linearized Monge--Amp\`ere equation in divergence form in this section.
Unless otherwise indicated, we assume $n\geq3$, \eqref{eq:det} and \eqref{eq:outer}. 
\subsection{Bounded vector fields}
For the fixed smooth
 strictly convex potential function $u$, it is known (see the argument in \cite[Remark~14.3]{Le24}) that  $W^{1, 2}$ weak solutions 
to the homogeneous linearized Monge--Amp\`ere equation \[U^{ij}D_{ij} w=0\quad\text{in}\quad S_u(z,r)\Subset\Omega\]
are locally in $W^{2,n}$. We use two
standard estimates.

\medskip

From the Caffarelli--Guti\'errez Harnack inequality
\cite[Theorem~12.4]{Le24}, we have the following oscillation estimate (see \cite[(12.26)]{Le24})
\begin{equation}
\label{eq:homosc}
\osc_{S_u(z,r/2)}w\leq\theta\,\osc_{S_u(z,r)}w,
       \qquad 0<\theta=\theta(n,\lambda,\Lambda)<1.
\end{equation}
Moreover, the
nonnegative subsolution estimate in \cite[Lemma~2.4]{CWZ25}, applied to
$|w|$, gives
\begin{equation}
   \label{eq:homLp}
 \sup_{S_u(z,r/2)}|w|\leq C_p
       \Bigg(\fint_{S_u(z,r)}|w|^p\dd\Bigg)^{1/p},\qquad C_p= C_p(n,\lambda,\Lambda, p).
    \end{equation}

From the inclusion property of interior sections \cite[Theorem~5.30(i)]{Le24}, we can find a constant $c_*=c_*(n,\lambda,\Lambda)>0$ such that as long as $S_{4h_0}=S_u(x_0, 4h_0)\Subset\Omega$,  $r_*=c_*h_0$
satisfies
\begin{equation}\label{eq:inclusions}
 S_u(z,4r_*)\Subset S_{3h_0/2} \quad \text{and}\quad
 S_u(x,4r_*)\Subset S_{5h_0/4} 
\end{equation}
for all $z\in S_{5h_0/4}$ and $x\in S_{h_0}$.

\begin{proof}[Proof of Theorem~\ref{thm:main}] In this proof, the constants $C$ depend only on $n,\lambda,\Lambda$, $h_0$, and $\diam(S_{4h_0})$.
Write \[F_0=\norm{\F}_{L^\infty(S_{2h_0})} \quad \text{and}\quad
V_p=\norm{v}_{L^p(S_{2h_0})}.\] For a fixed, strictly convex function $u\in C^3(\Omega)$, the classical local
elliptic regularity theory gives continuous representatives of the solutions used
below; see \cite[Chapter~8]{GT01}. This observation is qualitative only.
The constants in the following estimates are structural.

\smallskip\noindent
\emph{Step 1: Local boundedness.}
For $z\in S_{5h_0/4}$, 
let $w_z\in W^{1, 2}_0(S_u(z,2r_*))$ be the solution to
\[L_uw_z=\diver\F\quad \text{in }S_u(z,2r_*),\quad w_z=0\quad \text{on }\p S_u(z,2r_*).\]
 By \eqref{eq:inclusions}
and Corollary~\ref{cor:zero}, \[\norm{w_z}_{L^\infty(S_u(z,2r_*))}\leq Ch_0^\beta F_0.\]
The function $H_z:=v-w_z$ satisfies the homogeneous equation
\[L_uH_z=0\quad \text{in }S_u(z,2r_*).\]
Applying \eqref{eq:homLp}
on $S_u(z,r_*)$ and using \eqref{eq:volume}, we obtain
\[
 |v(z)|\leq C_p|S_u(z,r_*)|^{-1/p}
                   \norm{H_z}_{L^p(S_u(z,r_*))}+\norm{w_z}_{L^\infty(S_u(z,r_*))}
 \leq C_p\big(h_0^{-n/(2p)}V_p+h_0^\beta F_0\big).
\]
Consequently,
\begin{equation}\label{eq:bounded}
 \norm{v}_{L^\infty(S_{5h_0/4})}
 \leq C_p\big(h_0^{-n/(2p)}V_p+h_0^\beta F_0\big).
\end{equation}

\smallskip\noindent
\emph{Step 2: Oscillation decay.}
Fix $x\in S_{h_0}$ and $0<h\leq r_*$. 
On
$S=S_u(x,h)$, split \[v=H+w,\] where
\[L_uw=\diver\F\quad \text{in } S,\quad w=0 \quad \text{on }\partial S,\]
and 
\[L_uH=0\quad \text{in } S,\quad H=v \quad \text{on }\partial S.\]
By the weak maximum principle,
\[\osc_S H\le\osc_S v.\]  
From \eqref{eq:homosc} (applied to $H$) and \eqref{eq:zero} in Corollary~\ref{cor:zero}, we have
\begin{equation}\label{osc1}
 \osc_{S_u(x,h/2)}v\leq  \osc_{S_u(x,h/2)}H+  \osc_{S_u(x,h/2)}w\leq\theta \osc_{S_u(x,h)}v+Ch^\beta F_0.
\end{equation}
This is the splitting argument in the proof of
\cite[Theorem~15.6]{Le24}, with Corollary~\ref{cor:zero} supplying the uniform estimate for zero-boundary solutions. Choose
\[
 \kappa=\frac12\min\{\beta,-\log_2\theta\}>0.
\]
The geometric iteration used in \cite[proof of Theorem~12.13]{Le24}
gives
\begin{equation}\label{eq:decay}
 \osc_{S_u(x,h)}v\leq C\Big(\frac h{r_*}\Big)^\kappa
       \big(\osc_{S_u(x,r_*)}v+r_*^\beta F_0\big),\qquad 0<h\leq r_*.
\end{equation}
Note that equations \eqref{eq:inclusions} and \eqref{eq:bounded} bound the initial
oscillation in \eqref{eq:decay} 
\[\osc_{S_u(x,r_*)}v \leq C_p(V_p+F_0).\]

\smallskip\noindent
\emph{Step 3: H\"older continuity.}
By \eqref{eq:geometry}, for $x,y\in S_{h_0}$,
\[
 0\leq u(y)-u(x)-Du(x)\cdot(y-x)\leq B|x-y|^{1+\alpha},\quad B= B(n,\lambda,\Lambda, h_0, \diam(S_{4h_0})).
\]
If $x\ne y$ and
$h=2B|x-y|^{1+\alpha}\leq r_*$, then $y\in S_u(x,h)$, and
\eqref{eq:decay} yields
\[
 |v(x)-v(y)|\leq C_p(V_p+F_0)|x-y|^\gamma,
 \qquad \gamma=(1+\alpha)\kappa.
\]
If $h=2B|x-y|^{1+\alpha}\geq r_*$, then $|x-y|\geq [r_*/(2B)]^{\frac{1}{1+\alpha}}$,
so the same estimate follows from \eqref{eq:bounded}.
Since $0<\gamma\leq\alpha/2<1$ and $\gamma=\gamma(n,\lambda,\Lambda)$, this proves
\eqref{eq:main}.
\end{proof}

\begin{remark}
The proof does not use a higher Lebesgue norm of $D^2u$. Its new input
is Proposition~\ref{prop:green}.
\end{remark}

\subsection{\texorpdfstring{$L^{q}$}{Lq} vector fields}\label{sec:finite-Lq}

This section records a possible strengthening of Theorem~\ref{thm:main}
that follows from Proposition~\ref{prop:green-higher}.   The point is that
the gain from $D_xg_S\in L^{1+\varepsilon}$ allows the boundedness
assumption on the divergence datum to be replaced by sufficiently high,
but finite, Lebesgue integrability.

\medskip

Fix once and for all an exponent $\varepsilon_G>0$ for which
Proposition~\ref{prop:green-higher} holds, and set
\begin{equation}\label{eq:qstar-finite-data}
 q_*:=\frac{1+\varepsilon_G}{\varepsilon_G},
 \qquad
 \vartheta_G:=(1+\varepsilon_G)\beta-\frac n2\varepsilon_G>0.
\end{equation}
Then
\begin{equation}\label{eq:eta-star-finite-data}
 \frac{\vartheta_G}{1+\varepsilon_G}
 =\beta-\frac{n}{2q_*}>0.
\end{equation}
In particular, $q_*>n/(2\beta)>n$ because $\beta<1/2$.

\begin{proposition}[Uniform estimate for zero-boundary solutions with $L^q$ vector fields]
\label{prop:zero-finite-Lq} Assume $n\geq3$, \eqref{eq:det} and \eqref{eq:outer}. 
Let $S=S_u(z,h)$ satisfy $S_u(z,2h)\Subset S_{3h_0/2}$, and let
$q\in[q_*,\infty)$.  If $\F\in L^q(S;\R^n)$ and
$w\in W^{1,2}_0(S)$ solves
\[
 L_uw=\diver\F\qquad\text{in }S,\quad w=0\qquad\text{on }\p S,
\]
then
\begin{equation}\label{eq:zero-finite-Lq}
 \|w\|_{L^\infty(S)}
 \le C_q h^{\eta_q}\|\F\|_{L^q(S)},
 \qquad
 \eta_q:=\beta-\frac{n}{2q}>0,
\end{equation}
where $C_q=C(n,q,\lambda,\Lambda,h_0, \diam(S_{4h_0}))$.
\end{proposition}

\begin{proof}
Assume first that the vector field $\F$ is smooth.  For $y\in S$, 
the representation formula \eqref{Grep} can be applied to $w$, 
and integration by parts gives
\[
 w(y)=-\int_S g_S(x,y)\diver\F(x)\dd
      =\int_S D_xg_S(x,y)\cdot\F(x)\dd.\]
Here, $g_S(\cdot, y)$ is the Green's function of $L_u$ in $S$ with pole $y$. 

The H\"older inequality with exponents
$1+\varepsilon_G$ and $q_*$ and Proposition~\ref{prop:green-higher} give
\[
 |w(y)|
 \le \|D_xg_S(\cdot,y)\|_{L^{1+\varepsilon_G}(S)}
      \|\F\|_{L^{q_*}(S)} \leq C\|\F\|_{L^{q_*}(S)}h^{\vartheta_G/(1+\varepsilon_G)}.
\]
If $q\ge q_*$, the volume estimate for sections in \eqref{eq:volume} yields
\[
 \|\F\|_{L^{q_*}(S)}
 \le |S|^{\frac1{q_*}-\frac1q}\|\F\|_{L^q(S)}
 \le C h^{\frac n2(\frac1{q_*}-\frac1q)}
        \|\F\|_{L^q(S)}.
\]
Combining these estimates and recalling  \eqref{eq:eta-star-finite-data}, 
we obtain \eqref{eq:zero-finite-Lq} for smooth data.

\medskip

For a general vector field $\F\in L^q(S)$, one may extend $\F$ by zero and mollify to obtain smooth approximants $\F_k$.
Since $q\ge q_*>n>2$, $\F_k$ converges to $\F$ in $L^2(S)$,
so the corresponding solution to 
\[
 L_uw_k=\diver\F_k\qquad\text{in }S,\quad w_k=0\qquad\text{on }\p S,
\]
converges to $w$ in
$W^{1,2}_0(S)$ by the usual energy estimate.  The uniform
$L^\infty(S)$ bound on $w_k$ passes to an almost-everywhere convergent subsequence,
which proves \eqref{eq:zero-finite-Lq} for $w$.  \end{proof}

\begin{theorem}[Interior H\"older estimate with $L^q$ vector fields]
\label{thm:finite-Lq-holder}
Let $n\ge3$, and assume \eqref{eq:det} and \eqref{eq:outer}.  Let
$p\in(1,\infty)$ and $q\in[q_*,\infty)$.  Suppose that
$\F\in L^q(S_{2h_0};\R^n)$ and
$v\in W^{1,2}_{\loc}(S_{2h_0})\cap L^p(S_{2h_0})$ satisfies
\[
 L_uv=\diver\F
 \qquad\text{weakly in }S_{2h_0}.
\]
Then there exist
$\gamma_q=\gamma_q(n,q,\lambda,\Lambda)\in(0,1)$ and
$C=C(n,p,q,\lambda,\Lambda,h_0, \diam(S_{4h_0}))>0$ such that
\begin{equation}\label{eq:finite-Lq-holder}
 \|v\|_{L^\infty(S_{h_0})}
 +[v]_{C^{\gamma_q}(S_{h_0})}
 \le C\left(
 \|\F\|_{L^q(S_{2h_0})}
 +\|v\|_{L^p(S_{2h_0})}
 \right).
\end{equation}
\end{theorem}

\begin{proof}
The proof is similar to the proof of
Theorem~\ref{thm:main} where Corollary~\ref{cor:zero} is now replaced by 
Proposition~\ref{prop:zero-finite-Lq}.  We include the details to keep the dependence on
$q$ and the section scale explicit. We use the estimates \eqref{eq:homosc}-\eqref{eq:homLp} and inclusions \eqref{eq:inclusions} as in the last subsection. 

In this proof, the constants $C, C_p$ depend only on $n,\lambda,\Lambda, p, q$, $h_0$, and $\diam(S_{4h_0})$.
Write
\[
 F_q:=\|\F\|_{L^q(S_{2h_0})},
 \qquad
 V_p:=\|v\|_{L^p(S_{2h_0})}.
\]

For $z\in S_{5h_0/4}$, let $w_z\in W^{1, 2}_0(S_u(z,2r_*))$ be the solution to
\[L_uw_z=\diver\F\quad \text{in }S_u(z,2r_*),\quad w_z=0\quad \text{on }\p S_u(z,2r_*).\]  By
Proposition~\ref{prop:zero-finite-Lq},
\[
 \|w_z\|_{L^\infty(S_u(z,2 r_*))}
 \le Cr_*^{\eta_q}F_q.
\]
The function $H_z:=v-w_z$ 
 satisfies the homogeneous equation
\[L_uH_z=0\quad \text{in }S_u(z,2r_*).\]
  Applying
\eqref{eq:homLp} on $S_u(z,r_*)$ and using \eqref{eq:volume},
\begin{align*}
 |v(z)|
 \le C_p |S_u(z,r_*)|^{-1/p}
        \|H_z\|_{L^p(S_u(z,r_*))}
       +\|w_z\|_{L^\infty(S_u(z,r_*))}\le C\bigl(V_p+F_q\bigr),
\end{align*}
where $r_*^{-n/(2p)}$ is absorbed into the 
constant $C$.  Consequently,
\begin{equation}\label{eq:finite-Lq-bounded}
 \|v\|_{L^\infty(S_{5h_0/4})}
 \le C\bigl(V_p+F_q\bigr).
\end{equation}

Next fix $x\in S_{h_0}$ and $0<h\le r_*$.  Then, as in the derivation of \eqref{osc1} where
Corollary~\ref{cor:zero} is now replaced by 
Proposition~\ref{prop:zero-finite-Lq}, we have
\begin{equation}\label{eq:finite-Lq-recurrence}
 \osc_{S_u(x,h/2)}v
 \le\theta \osc_{S_u(x,h)}v+C h^{\eta_q}F_q.
\end{equation}
Since $\eta_q>0$, choose
\[
 \kappa_q:=\frac12\min\{\eta_q,-\log_2\theta\}>0.
\]
The standard geometric iteration then yields
\begin{equation}\label{eq:finite-Lq-decay}
 \osc_{S_u(x,h)}v
 \le C\left(\frac h{r_*}\right)^{\kappa_q}
 \left(\osc_{S_u(x,r_*)}v+r_*^{\eta_q}F_q\right),
 \qquad 0<h\le r_*.
\end{equation}
By \eqref{eq:finite-Lq-bounded}, the quantity in parentheses is bounded by
$C(V_p+F_q)$.

\medskip

As in {\it Step 3} in the proof of
Theorem~\ref{thm:main}, we have \eqref{eq:finite-Lq-holder}
with
$\gamma_q=(1+\alpha)\kappa_q\in(0,1)$.
Indeed, $\eta_q\leq\beta$, hence
$\gamma_q\le(1+\alpha)\beta/2=\alpha/2<1$.
\end{proof}

\begin{remark}
\label{LqFrem}
The above extension of the interior H\"older estimate in  Theorem~\ref{thm:main} to $L^q$ vector fields uses the higher-integrability theory for
the Hessian $D^2u$ through Proposition~\ref{prop:green-higher}.  Thus, it is logically
separate from Theorem~\ref{thm:main}, whose proof for bounded vector fields does not use
the interior $W^{2,1+\varepsilon}$ estimate of the Monge--Amp\`ere equation.
\end{remark}

\section{Global H\"older estimates}\label{sec:global}

We study global H\"older estimates for the linearized Monge--Amp\`ere equation in divergence form in this section under the following structural  conditions.

\medskip

\paragraph{\bf Global structural assumptions.} Let $\Omega \subset \mathbb{R}^n$ be a convex domain and assume that there exists a constant $\rho>0$ such that
\begin{equation}\label{glb1}
     \Omega \subset B_{1 / \rho}(0) \subset \mathbb{R}^n
\end{equation}
and for each $y \in \partial \Omega$, 
\begin{equation}\label{glb2}
   \text{there is an interior ball } B_\rho(z) \subset \Omega \text{ such that } y \in \partial B_\rho(z). 
\end{equation}
Let $u \in C^{1,1}(\overline{\Omega}) \cap C^3(\Omega)$ be a convex function satisfying
\begin{equation}\label{glb3}
    0<\lambda \leq \det D^2 u \leq \Lambda \quad \text { in } \Omega.\noeqref{glb2,glb3}
\end{equation}
Assume further that on $\partial \Omega, u$ separates quadratically from its tangent hyperplanes; namely, for all $x_0, x \in \partial \Omega$, we have
\begin{equation}\label{glb4}
    \rho\left|x-x_0\right|^2 \leq u(x)-u\left(x_0\right)-D u\left(x_0\right) \cdot\left(x-x_0\right) \leq \rho^{-1}\left|x-x_0\right|^2.
\end{equation}
For the boundary localization framework associated with this hypothesis, see Savin \cite{Sav13}.

For an open, proper set $V\subset\Omega$, let $g_V(\cdot, y)$ be the Green's function of the linearized Monge--Amp\`{e}re operator $L_u=D_i(U^{i j} D_j)$ in $V$ with pole $y\in V$.
\subsection{Global integrability of the gradient of the Green's function}
We first recall \cite[Lemma 3.5]{GL26}.
\begin{lemma} 
\label{mubound}
Assume that $u$ and $\Omega$ satisfy \eqref{glb1}--\eqref{glb4}. There exists $\alpha_b=\alpha_b(n,\lambda,\Lambda)\in (0, 1)$ such that the following statements hold.
	\begin{enumerate}
	\item[(i)] There exists $C_\ast=C_\ast(n, \lambda, \Lambda, \rho)>0$ such that 
	 \begin{equation}
\label{Duglb}
[Du]_{C^{\alpha_b}(\overline\Omega)}:=\sup_{x\neq y\in \overline\Omega}|Du(x)-Du(y)|/|x-y|^{\alpha_b}\leq C_\ast.
\end{equation}
\item[(ii)] Let $S_u(x_{0}, t_{0})$ be a section of $u$ with $x_0\in \overline{\Omega}$ and $t_{0}>0$. Then
\begin{equation}
\label{Hausdorffest}
\mathcal{H}^{n-1} (\partial S_u(x_{0}, t_{0}))\leq C(n, \lambda, \Lambda,  \rho) t_0^{\frac{n-2}{2} + \frac{
\alpha_b}{\alpha_b +1}}.\end{equation}
	\end{enumerate}
\end{lemma}
The following lemma is a global version of Lemma \ref{lem:mass}.
\begin{lemma}[Hessian mass]
\label{lem:massbdr}
Assume that $u$ and $\Omega$ satisfy \eqref{glb1}--\eqref{glb4}. Let $S_u(x_{0}, t_{0})$ be a section of $u$ with $x_0\in \overline{\Omega}$ and $t_{0}>0$. Then
\begin{equation}\label{eq:massbdr}
 \int_{S_u(x_{0}, t_{0})}\Delta u\dd\leq C t_0^{\frac{n-2}{2} + \frac{
\alpha_b}{\alpha_b +1}},\qquad C=C(n, \lambda, \Lambda,  \rho).
\end{equation}
\end{lemma}
\begin{proof}
By the divergence theorem, we have
\begin{equation*}
\begin{split}
 \int_{S_u(x_{0}, t_{0})}\Delta u\dd &=  \int_{S_u(x_{0}, t_{0})}\Delta [u(x)-u(x_0)-Du(x_0)\cdot (x-x_0)]\dd\\
 &= \int_{\partial S_u(x_{0}, t_{0})} [Du(x)-Du(x_0)]\cdot\nu \,\mathrm{d}S.
 \end{split}
\end{equation*}
The claimed estimate now follows from Lemma \ref{mubound}, noting that
\[ |Du(x)-Du(x_0)|\leq C_\ast |x-x_0|^{\alpha_b} \leq C_\ast (2/\rho)^{\alpha_b},\quad \text{for all } x\in \partial S_u(x_0, t_0).\]
\end{proof}

\begin{remark}
\label{rem:bdrinfo}
Observe that, under the assumptions \eqref{glb1}--\eqref{glb4}, by \cite[Lemma 9.7]{Le24}, there exists $M=M(n,\lambda,\Lambda,\rho)>0$ such that 
\begin{equation}
\label{Omsec}
\overline{\Omega}\subset S_u(x_0, M) \quad\text{for all}\quad x_0\in\overline{\Omega}.\end{equation} Therefore, $\Omega$ is a section of $u$ with height comparable to $1$. Consider an open, proper subset $V\subset\Omega$. By the global localization of the superlevel sets of the Green's function of the Monge--Amp\`ere operator in
\cite[Lemma~4.2(i)]{GL26},  there exist constants $C_g,T_g>0$ depending only on $n,\lambda,\Lambda, \rho$
such that, for  every $y\in V$, and every
$k\geq T_g$,
\begin{equation}\label{eq:global-level-localization-Lq}
 \{x\in V:g_V(x,y)>k\}
 \subset S_u\bigl(y,C_g k^{-\frac{2}{n-2}}\bigr).
\end{equation}
Moreover, if
$g=g_V(\cdot,y)$, then
\begin{equation}\label{eq:global-energy-Lq}
 |Dg|^2\le\lambda^{-1}\Delta u\,U^{ij}D_i gD_j g,
 \qquad
 \int_{\{g\le k\}}U^{ij}D_i gD_j g\,\mathrm{d}x=k.
\end{equation}
\end{remark}
As in the proof of Proposition \ref{prop:green} (without using
the global $W^{2,1+\varepsilon}$-estimate for the Monge--Amp\`ere equation), we have 
\begin{proposition}\label{prop:greenbdr}
Let $n\geq3$, and assume that $u$ and $\Omega$ satisfy \eqref{glb1}--\eqref{glb4}. Let $S:=S_u(x_{0}, t_{0})$ be a section of $u$ with $x_0\in \overline{\Omega}$ and $t_{0}>0$.
Then
\begin{equation}\label{eq:greenbdr}
 \sup_{y\in S}\int_S|D_xg_S(x,y)|\dd\leq 
 C t_0^{ \frac{
\alpha_b}{2(\alpha_b +1)}},\qquad C=C(n, \lambda, \Lambda,  \rho).
\end{equation}
In particular,
\begin{equation}\label{eq:green-global}
 \sup_{y\in\Omega}\int_\Omega|D_xg_\Omega(x,y)|\dd
 \leq C(n,\lambda,\Lambda,\rho).
\end{equation}
\end{proposition}

%%%%%%%%%%%%%%%%%%%%%%%%%%%%%%%%%%%%%%%%%%

Throughout Section \ref{sec:global}, we
assume the global structural hypotheses \eqref{glb1}--\eqref{glb4}, and 
write
\[
 a=\frac{n-2}{2},\qquad
 \beta_{\partial}:=\frac{\alpha_b}{1+\alpha_b},
\]
where $\alpha_b$ is the exponent in Lemma~\ref{mubound}.

\medskip
We first upgrade the global $L^1$ estimate for the gradient of the Green's function.  The point is that the
Hessian mass decay with respect to the section height in Lemma~\ref{lem:massbdr}, together with the known global
higher integrability of $D^2u$, yields a positive gain over $L^1$ on every
level band of the Green's function.  We use the global higher-integrability estimate in
\cite[Theorem~10.1]{Le24}; the same estimate is also recorded in the proof of
\cite[Theorem~2.1, Step~1]{Le20} under the present structural assumptions.
Thus, for some
\begin{equation}\label{eq:global-delta0}
 \delta_0=\delta_0(n,\lambda,\Lambda)>0,
 \qquad
 \int_\Omega (\Delta u)^{1+\delta_0}\,\mathrm{d}x
 \le C(n,\lambda,\Lambda,\rho).
\end{equation}
Interpolating \eqref{eq:global-delta0} with Lemma~\ref{lem:massbdr}, we obtain,
for every $0<\delta<\delta_0$ and every section $S_u(\xi,s)$,
\begin{equation}\label{eq:global-hessian-interpolation}
 \int_{S_u(\xi,s)}(\Delta u)^{1+\delta}\,\mathrm{d}x
 \le C_\delta
 s^{(a+\beta_{\partial})(1-\delta/\delta_0)}.
\end{equation}
Indeed, the interpolation parameter $\theta$ determined by
\[
 \frac1{1+\delta}=\theta+\frac{1-\theta}{1+\delta_0}
\]
satisfies $\theta(1+\delta)=1-\delta/\delta_0$, which gives
\eqref{eq:global-hessian-interpolation}.

\begin{proposition}[Global higher integrability of the gradient of the Green's function]
\label{prop:green-global-higher}
There exists
$\varepsilon_b=\varepsilon_b(n,\lambda,\Lambda)>0$
such that the following holds.  If $V\subset\Omega$ is open and proper (including $V=\Omega$) and
$g_V(\cdot,y)$ is the Dirichlet Green function of $L_u$ in $V$, then
\begin{equation}\label{eq:green-global-higher}
 \sup_{y\in V}\int_V|D_xg_V(x,y)|^{1+\varepsilon_b}\,\mathrm{d}x
 \le C(n,\lambda,\Lambda,\rho).
\end{equation}
\end{proposition}

\begin{proof} We will use \eqref{eq:global-level-localization-Lq} and \eqref{eq:global-energy-Lq}.
Let $0<\varepsilon_b<1/2$ to be determined with
\begin{equation}\label{eq:global-delta-eps}
 \delta=\frac{2\varepsilon_b}{1-\varepsilon_b}<\delta_0.
\end{equation}
Applying the H\"older inequality exactly as
in \eqref{eq:pair-eps-green-higher} gives for every measurable $E\subset V$
\begin{equation}\label{eq:global-pair-eps}
 \int_E|Dg|^{1+\varepsilon_b}\,\mathrm{d}x
 \le C
 \left(\int_E U^{ij}D_i gD_j g\,\mathrm{d}x\right)^{\frac{1+\varepsilon_b}{2}}
 \left(\int_E(\Delta u)^{1+\delta}\,\mathrm{d}x\right)^{\frac{1-\varepsilon_b}{2}}.
\end{equation}

Let $T_g$ be a fixed structural constant as in Remark \ref{rem:bdrinfo} and set
\[
 K_j=2^jT_g,\qquad
 E_j=\{K_j<g\le2K_j\},\qquad
 s_j=C_gK_j^{-1/a}.
\]
By \eqref{eq:global-level-localization-Lq},
$E_j\subset S_u(y,s_j)$, while \eqref{eq:global-energy-Lq} gives
$\int_{E_j} U^{ik}D_i gD_k g\le2K_j$.  Hence
\eqref{eq:global-hessian-interpolation} and
\eqref{eq:global-pair-eps} yield
\begin{equation}\label{eq:global-band-eps}
 \int_{E_j}|Dg|^{1+\varepsilon_b}\,\mathrm{d}x
 \le C s_j^{\Theta_b},
\end{equation}
where
\begin{align}
 \Theta_b
 =-\frac{a(1+\varepsilon_b)}2
   +\frac{1-\varepsilon_b}{2}(a+\beta_{\partial})
       \left(1-\frac{\delta}{\delta_0}\right)=\frac{\beta_{\partial}}2
   -\varepsilon_b\left(
       a+\frac{\beta_{\partial}}2
       +\frac{a+\beta_{\partial}}{\delta_0}
     \right).\label{eq:Theta-boundary}
\end{align}
Since $\Theta_b\to\beta_{\partial}/2>0$ as
$\varepsilon_b\to0^+$, we may choose
$\varepsilon_b=\varepsilon_b(n,\lambda,\Lambda)>0$ so small that both
\[\delta<\delta_0\quad \text{and}\quad \Theta_b>0.\]  Then
$s_j\approx2^{-j/a}$ and the right-hand side of
\eqref{eq:global-band-eps} is summable in $j$.

For the sublevel set $\{g\le T_g\}$, we use
\eqref{eq:global-energy-Lq}, \eqref{eq:global-delta0}, and
\eqref{eq:global-pair-eps} to obtain a uniform bound
\[\int_{\{g\leq T_g\}} |Dg|^{1+\varepsilon_b}\,\mathrm{d}x \leq C.\]
Summing \eqref{eq:global-band-eps} over $j$ and adding the resulting estimate to the preceding estimate, one obtains
\eqref{eq:green-global-higher} as asserted.  
\end{proof}

\subsection{A global H\"older estimate with  \texorpdfstring{$L^q$}{Lq} vector fields}\label{sec:global-finite-Lq}
We record here a global counterpart of the uniform estimate with $L^q$ vector fields from
Section~\ref{sec:finite-Lq}.  As in that section, the argument uses higher
integrability of $D^2u$, and is therefore logically separate from the
bounded-data proof of Theorem~\ref{thm:main}.

Fix an $\varepsilon_b= \varepsilon_b(n,\lambda,\Lambda)>0$ as in Proposition~\ref{prop:green-global-higher} and define
\begin{equation}\label{eq:qb-global}
 q_b:=\frac{1+\varepsilon_b}{\varepsilon_b}.
\end{equation}
After decreasing $\varepsilon_b$ if necessary, we may and do assume
$q_b>n$.  For $q>q_b$, set
\begin{equation}\label{eq:kappa-global-data}
 \kappa_q:=\frac1{q_b}-\frac1q>0,
 \qquad
 \eta_q:=\frac n2\kappa_q>0.
\end{equation}

\begin{lemma}[Maximum principle with $L^q$ vector fields]
\label{lem:global-Lq-maximum}
Assume $n\ge3$ and that $u$ and $\Omega$ satisfy
\eqref{glb1}--\eqref{glb4}.
Let  $V\subset\Omega$ be open and proper. Let $q>q_b$ and
$\F\in L^q(V;\mathbb R^n)$.  If
$w\in W^{1,2}_0(V)$ solves
\begin{equation}
\label{wVeq}
 L_uw=\diver\F\qquad\text{in }V,\quad w=0\qquad\text{on }\p V,
\end{equation}
then, there exists $C=C(n,\lambda,\Lambda,\rho, q)$ such that
\begin{equation}\label{eq:global-zero-Lq}
 \|w\|_{L^\infty(V)}
 \le C|V|^{\kappa_q}\|\F\|_{L^q(V)}.
\end{equation}
Consequently, if $v\in W^{1,2}(V)\cap C(\overline V)$ solves
$L_uv=\diver\F$ in $V$, then
\begin{equation}\label{eq:global-max-Lq}
 \|v\|_{L^\infty(V)}
 \le \|v\|_{L^\infty(\partial V)}
     +C|V|^{\kappa_q}\|\F\|_{L^q(V)}.
\end{equation}
\end{lemma}

\begin{proof}
As in the proof of Proposition \ref{prop:zero-finite-Lq}, it suffices to prove \eqref{eq:global-zero-Lq}
for smooth vector fields $\F$. For $y\in V$, 
the representation formula \eqref{Grep} can be applied to $w$, 
and integration by parts gives
\[
 w(y)=-\int_V g_V(x,y)\diver\F(x)\dd
      =\int_V D_xg_V(x,y)\cdot\F(x)\dd.\]
Applying the H\"older inequality twice and using Proposition~\ref{prop:green-global-higher}, we find
\[
  \|w\|_{L^\infty(V)}
 \le \sup_{y\in V}\|D_xg_V(\cdot,y)\|_{L^{1+\e_b}(V)}
      \|\F\|_{L^{q_b}(V)} \leq C \|\F\|_{L^{q_b}(V)} \le C|V|^{\frac1{q_b}-\frac1q}\|\F\|_{L^q(V)}.
      \]
For
\eqref{eq:global-max-Lq}, let $w\in W^{1, 2}_0(V)$ be the solution  to \eqref{wVeq}.  Then
$L_u(v-w)=0$ in $V$, so the weak maximum principle gives
\[\|v-w\|_{L^\infty(V)}\le\|v\|_{L^\infty(\partial V)}.\]
This combined with \eqref{eq:global-zero-Lq} proves \eqref{eq:global-max-Lq}.
\end{proof}

Lemma~\ref{lem:global-Lq-maximum} can be applied to $V=\Omega\cap S_u(x, h)$, and in this case,
using the upper bound $|S_u(x,h)|\le C(n,\lambda) h^{n/2}$ for the volume of sections \cite[Lemma 5.8]{Le24},
one has 
\begin{equation}\label{eq:global-section-zero-Lq}
 \|w\|_{L^\infty(\Omega\cap S_u(x, h))}
 \le C h^{\eta_q}\|\F\|_{L^q(S_u(x,h))}.
\end{equation}
Thus, the usual splitting argument gives the following
interior H\"older estimate.
\begin{lemma}[Interior H\"older estimate with $L^q$ vector fields] 
\label{intHLq}
Assume $n\ge3$ and that $u$ and $\Omega$ satisfy
\eqref{glb1}--\eqref{glb4}.
There exists
\begin{equation}\label{eq:global-kappai}
 0<\gamma_i<\frac12\min\{\eta_q,-\log_2\theta\},
\end{equation}
where $\theta\in (0, 1)$ is the constant in
\eqref{eq:homosc}, such that if \[L_u v=\diver \F\quad\text{in}\quad S_u(y,H)\subset\Omega,\] then for
$z_1,z_2\in S_u(y,H/8)$,
\begin{equation}\label{eq:global-scale-interior-Lq}
 |v(z_1)-v(z_2)|
 \le C_\ast\Big(
       \|v-c\|_{L^\infty(S_u(y,H/2))}
       +H^{\eta_q}\|\F\|_{L^q(\Omega)}
      \Big)
 H^{-\gamma_i}M_0^{\gamma_i}|z_1-z_2|^{\gamma_i},
\end{equation}
for every constant $c$, where $C_\ast$ depends only on $n,\lambda,\Lambda, \rho, q$ and, in view of Lemma~\ref{mubound},
\[M_0:=\sup_{x,x'\in\overline\Omega}|Du(x)-Du(x')|\le C=C(n,\lambda,\Lambda, \rho).\] 
\end{lemma}
Indeed, as in the proof of \cite[Theorem~6.2]{GL26}, estimate \eqref{eq:global-scale-interior-Lq} follows
by repeating the oscillation iteration
\eqref{eq:finite-Lq-recurrence}--\eqref{eq:finite-Lq-decay} with
\eqref{eq:global-section-zero-Lq} in place of
\eqref{eq:zero-finite-Lq}.

\medskip
We next record the boundary estimate.  Its proof is the barrier argument of
\cite[Proposition~7.1]{GL26}, with Lemma~\ref{lem:global-Lq-maximum}
replacing the maximum-principle estimate for the inhomogeneous term.

\begin{proposition}[Boundary H\"older estimate]
\label{prop:boundary-finite-Lq}
Assume $n\ge3$ and that $u$ and $\Omega$ satisfy
\eqref{glb1}--\eqref{glb4}.
Let $q>q_b$ and $\sigma\in(0,1]$.  Suppose that
$\F\in L^q(\Omega;\mathbb R^n)$ and
$v\in W^{1,2}(\Omega)\cap C(\overline\Omega)$ satisfies
\[
 L_uv=\diver\F\quad\text{in }\Omega,
 \qquad
 v=\phi\quad\text{on }\partial\Omega,
\]
with $\phi\in C^\sigma(\partial\Omega)$.  Set
\begin{equation}\label{eq:boundary-exponents-Lq}
 \sigma_0:=\min\left\{\sigma,\frac{n\kappa_q}{2}\right\},
 \qquad
 \sigma_1:=\frac{\sigma_0}{\sigma_0+3n}.
\end{equation}
Then there are $r_0,C>0$, depending only on
$n,\lambda,\Lambda,\rho,q$, and $\sigma$, such that
\begin{equation}\label{eq:boundary-holder-Lq}
 |v(x)-v(x_0)|
 \le C\bigl(\|\phi\|_{C^\sigma(\partial\Omega)}
             +\|\F\|_{L^q(\Omega)}\bigr)|x-x_0|^{\sigma_1}
\end{equation}
whenever $x_0\in\partial\Omega$ and
$x\in\Omega\cap B_{r_0}(x_0)$.
\end{proposition}

\begin{proof} The proof closely follows the proofs of \cite[Proposition 14.32]{Le24} and \cite[Proposition~7.1]{GL26} using maximum principles and barriers. Since our setting is different, we sketch the details for the reader's convenience. 

\medskip
By \eqref{eq:global-max-Lq} with $V=\Omega$,
\[
 \|v\|_{L^\infty(\Omega)}
 \le C\bigl(\|\phi\|_{C^\sigma(\partial\Omega)}
             +\|\F\|_{L^q(\Omega)}\bigr),\quad C=C(n,\lambda,\Lambda,\rho, q).
\]
After normalization, we may assume that 
\[\|v\|_{L^\infty(\Omega)} +\|\phi\|_{C^\sigma(\partial\Omega)}
             +\|\F\|_{L^q(\Omega)}=1.\]
Fix $x_0\in\partial\Omega$. Without loss of generality, we assume $x_0 = 0$, $u(0) = 0$, $Du(0) = 0$, the tangent hyperplane to $\overline{\Omega}$ at \(x_0=0\) is \(\{x_n=0\}\), and \(\Omega\subset\{x_n>0\}\).

As in the proof of \cite[Proposition~7.1]{GL26}, for any $\varepsilon \in (0, 1)$,  let
	\[
		h_\pm:= v - v(0) \pm \varepsilon \pm \frac{6}{\delta_2^3}w_{\delta_2} \; \quad \text{ in } A:= \Omega\cap B_{\delta_2}(0),
	\]
	where $\delta_2\in (0, 1)$ small is to be chosen later and the function $w_{\delta_2}$ is defined by 
	\[
		w_{\delta_2}(x) = w_{\delta_2}(x', x_n):= M_{\delta_2}x_n + u(x) - \tilde{\delta_2}|x'|^2 - \Lambda^n (\lambda \tilde{\delta_2})^{1-n} x_n^2\quad \text{ for } x=(x', x_n) \in \overline{\Omega},
	\]
	where \[\tilde{\delta_2}:= \delta_2^3/2 \quad\text{and}\quad  M_{\delta_2}:= \Lambda^n(\lambda \tilde{\delta_2})^{1-n}.\] 
If $\delta_2\le\varepsilon^{1/\sigma_0}$, then one can verify that
\[L_uh_+\le\diver\F \quad\text{and} \quad L_uh_-\ge\diver\F \quad\text{in}\quad A;\quad h_- \leq 0\quad \text{ and } h_+ \geq 0\quad \text{ on } \partial A. \] Let $z_A\in W^{1, 2}_0(A)$ solve
\[L_uz_A=\diver\F \quad\text{in}\quad A,\quad z_A=0 \quad \text{ on } \partial A. \]  Applying the weak
maximum principle to $h_+-z_A$ and $h_--z_A$, and then
Lemma~\ref{lem:global-Lq-maximum}, gives
\begin{equation}\label{eq:barrier-error-Lq}
 \max\{-h_+,h_-\}
 \le \|z_A\|_{L^\infty(A)}
 \le C|A|^{\kappa_q}\|\F\|_{L^q(\Omega)}
 \le C\delta_2^{n\kappa_q}.
\end{equation}
The formula for $w_{\delta_2}$ and the boundary estimates of $u$ give (see \cite[inequality (7.3)]{GL26})
\[
 \frac6{\delta_2^3}w_{\delta_2}(x)
 \le C\delta_2^{-3n}|x-x_0|.
\]
Choose $\delta_2=\varepsilon^{1/\sigma_0}$.  Since
$\sigma_0\le n\kappa_q/2$, the last term in
\eqref{eq:barrier-error-Lq} is bounded by $C\varepsilon^2$, hence by
$\varepsilon$ after restricting to sufficiently small $\varepsilon$.
Consequently,
\[
 |v(x)-v(x_0)|
 \le2\varepsilon+C\varepsilon^{-3n/\sigma_0}|x-x_0|.
\]
Taking
$\varepsilon=|x-x_0|^{\sigma_0/(\sigma_0+3n)}$ proves
\eqref{eq:boundary-holder-Lq}.  The proof is complete.
\end{proof}

\begin{proof}[Proof of Theorem \ref{thm:global-finite-Lq}]
Note that our assumptions imply \eqref{glb1}--\eqref{glb4}, where $\rho$ now depends only on $n, \lambda, \Lambda, \|u\|_{C^3(\partial \Omega)}$, $R$, and the $C^3$ regularity of $\partial \Omega$; see \cite[Proposition 4.7]{Le24}. 

\medskip
As in the proof of \cite[Theorem 2.4]{GL26}, we combine Proposition~\ref{prop:boundary-finite-Lq}, the 
interior H\"older estimate \eqref{eq:global-scale-interior-Lq} in Lemma \ref{intHLq}, and the geometry of
maximal interior sections.  We sketch the details for the reader's convenience.

\medskip
For points having at least a fixed positive distance from $\partial\Omega$, the maximal
interior section has a fixed positive height, and
\eqref{eq:global-scale-interior-Lq} gives a uniform interior H\"older
estimate.  

Consider therefore $y\in\Omega$ with
$r=\operatorname{dist}(y,\partial\Omega)$ small.  Let $S_u(y,\bar h)$ be its maximal
interior section and let $y_0\in\partial\Omega\cap\partial S_u(y,\bar h)$.
From Savin's boundary localization theorem \cite[Theorem 3.1]{Sav13}, the maximal interior section $S_u(y,\bar h)$ has the following properties
(see \cite[Proposition~7.2]{GL26}),
\begin{equation}\label{eq:maximal-section-global-Lq}
 c r^2\le\bar h\le Cr^2,
 \qquad
 S_u(y,\bar h)\subset B_{Cr|\log r|}(y),
 \qquad
 B_{cr^2}(y)\subset S_u(y,\bar h/8).
\end{equation}
For small $r$, Proposition~\ref{prop:boundary-finite-Lq} therefore gives
\begin{equation}\label{eq:maximal-v-bound-Lq}
 \|v-v(y_0)\|_{L^\infty(S_u(y,\bar h/2))}
 \le C(r|\log r|)^{\sigma_1}
 \Big(\|\phi\|_{C^\sigma(\partial\Omega)}
       +\|\F\|_{L^q(\Omega)}\Big).\noeqref{eq:maximal-v-bound-Lq}
\end{equation}

In \eqref{eq:global-kappai}, decrease $\gamma_i$ further, if necessary,
so that
\begin{equation}\label{eq:global-gamma-choice}
 \gamma_i<\min\left\{\frac{\sigma_1}{4},\frac{\eta_q}{2}\right\}.
\end{equation}
Applying \eqref{eq:global-scale-interior-Lq} on the maximal section 
$S_u(y,\bar h)$
with
$c=v(y_0)$ and $H=\bar h$, we find
 \[|v(z_1)-v(z_2)|
 \le C\Big[(
       r|\log r|)^{\sigma_1}
 \Big(\|\phi\|_{C^\sigma(\partial\Omega)}
       +\|\F\|_{L^q(\Omega)}\Big)
       +{\bar h}^{\eta_q}\|\F\|_{L^q(\Omega)}
      \Big]
 {\bar h}^{-\gamma_i}M_0^{\gamma_i}|z_1-z_2|^{\gamma_i}.\]
From
\eqref{eq:maximal-section-global-Lq}--\eqref{eq:global-gamma-choice}, we have
\[
 (r|\log r|)^{\sigma_1}\bar h^{-\gamma_i}
 +\bar h^{\eta_q-\gamma_i}
 \le C.
\]
Thus, we can find a constant $C=C(n,\lambda,\Lambda, \rho, q,\sigma)$ such that
\begin{equation}\label{eq:local-near-boundary-global-Lq}
 |v(z_1)-v(z_2)|
 \le C\Big(\|\phi\|_{C^\sigma(\partial\Omega)}
       +\|\F\|_{L^q(\Omega)}\Big)|z_1-z_2|^{\gamma_i}\quad\text{for all } z_1,z_2\in S_u(y,\bar h/8).
\end{equation}

Now, as in Step 3 in the proof of \cite[Theorem 2.4]{GL26}, we obtain 
\eqref{eq:global-finite-Lq-holder} with, for example,
\[
 \gamma=\min\left\{\gamma_i,\frac{\sigma_1}{2}\right\}>0.\]
 Here is the basic idea. If
$x,y\in\overline\Omega$ are sufficiently close and
$|x-y|<c\operatorname{dist}(x,\partial\Omega)^2$ (where we can assume, without loss of generality, that $\operatorname{dist}(x,\partial\Omega)\le\operatorname{dist}(y,\partial\Omega)$), then
\eqref{eq:local-near-boundary-global-Lq} applies.  Otherwise both points
lie within a distance of order $C|x-y|^{1/2}$ from the boundary, and
Proposition~\ref{prop:boundary-finite-Lq}, together with the
$C^\sigma$ modulus of $\phi$, gives
\[
 |v(x)-v(y)|\le C\Big(\|\phi\|_{C^\sigma(\partial\Omega)}
       +\|\F\|_{L^q(\Omega)}\Big)|x-y|^{\sigma_1/2}.
\]
If
$x,y\in\overline\Omega$ are separated by a fixed distance, the global $L^\infty$ estimate
from Lemma~\ref{lem:global-Lq-maximum} suffices. 
The theorem is proved.
\end{proof}

%%%%%%%%%%%%%%%%%%%%%%%%%%%%%%%%%%%%%%%%%%
\section{Applications: semigeostrophic equations and Monge--Amp\`ere homogenization}
\label{sec:applications}

This section presents two applications of our H\"older estimates.
\subsection{Application to the three-dimensional dual semigeostrophic equations}\label{sec:sg}

In this section, we apply Theorem~\ref{thm:main} to the dual
semigeostrophic equations on the three-dimensional torus $\mathbb T^3$.
As introduced above, the dual formulation is
\begin{equation}\label{eq:dual-SG}
 \left\{
 \begin{aligned}
 \partial_t\rho_t(x)+\operatorname{div}(\rho_t(x)U_t(x))&=0,
     &&(t,x)\in(0,\infty)\times\mathbb T^3,\\
 U_t(x)&=J\bigl(x-D P_t^*(x)\bigr),
     &&(t,x)\in(0,\infty)\times\mathbb T^3,\\
 \det D^2P_t^*(x)&=\rho_t(x),
     &&(t,x)\in(0,\infty)\times\mathbb T^3,\\
 P_t^*&\text{ is convex},\\
 \rho_0(x)&=\rho^0(x),
     &&x\in\mathbb T^3,
 \end{aligned}
 \right.
\end{equation}
with the periodicity condition
\begin{equation}\label{eq:period}
 P_t^*(x)-\frac12|x|^2\quad\text{is $\mathbb Z^3$-periodic.}
\end{equation}
Here $J$ is the matrix appearing in \eqref{eq:SG}. Equivalently,
$U_t=(x-D P_t^*)^\perp$, where
$w^\perp=(-w_2,w_1,0)$ for $w=(w_1, w_2, w_3)\in\R^3$. The initial density $\rho^0$ is a probability
density on $\mathbb T^3$.

\medskip

Global weak solutions for semigeostrophic systems in dual variables were
constructed in \cite{BB,CuG,LoN}; see also \cite{ACDF14} for the
three-dimensional theory. When the initial density is H\"older continuous
and uniformly bounded away from zero and infinity, Loeper \cite{Lo06}
proved short-time existence and uniqueness in a more regular class. As in
Subsection \ref{subsec:intro-dual-semi-eq}, our concern here is the time regularity under only the
two-sided density bounds.

\medskip

We now prove Theorem~\ref{thm:holder-dual}. The argument follows
\cite[Section~4]{Le18} (see also \cite[Theorems 15.12 and 15.13]{Le24}). The only substantive change is that the
three-dimensional equation for $\partial_tP_t^*$ is now covered by
Theorem~\ref{thm:main}; the estimates used to control the dual flow are the
three-dimensional analogues from \cite{ACDF14}.

\begin{proof}[Proof of Theorem~\ref{thm:holder-dual}] We use an approximation argument as in \cite{ACDF, Lo05}. That is, extending $\rho_t=\rho_0$ and $U_t=0$ for $t\leq 0$, 
 we first regularize the density $\rho_t$ by $\rho_t^\e=\rho_t * \eta_\e$ with $\eta_\e$ being a nonnegative standard space-time mollifier   and the vector field $U_t$ by 
$U_t^\e= ((\rho_tU_t)*\eta_\e)/\rho_t^\e$. Next, we solve the corresponding periodic Monge--Amp\`ere equations $\det D^2 P_t^{\ast, \e}= \rho_t^\e$ on $\mathbb T^3$.  
We then establish the bounds in $L^{\infty}((0,\infty); C^{\alpha}(\mathbb T^3))$ for $\p_t P_t^{\ast,\e}$ and $\p_t P_t^\e$ with bounds depending only on $\lambda$ and $\Lambda$. Finally, we pass to the limit $\e\to 0$.
Therefore, we can
assume from now on that  $\rho_t, U_t$, $P_t^{\ast}$ and $P_t$ are smooth. 

First, the density bounds are propagated by the dual flow. In particular,
by \cite[Lemma~9.1]{Lo05},
\begin{equation}\label{eq:bd-rho_t}
 \lambda\leq\rho_t\leq\Lambda
 \qquad\text{in }\mathbb T^3
 \qquad \text{for all }t\geq0.
\end{equation}
As in \eqref{eq:parP-intro}, differentiating the Monge--Amp\`ere equation in
time gives
\begin{equation}\label{eq:parP}
 \operatorname{div}\!\left(
   \rho_t(D^2P_t^*)^{-1} D\partial_tP_t^*
 \right)
 =\operatorname{div}F_t,
 \qquad F_t=-\rho_tU_t.
\end{equation}
The periodic transport estimates (see, for example, \cite[Theorem~1]{CE}) imply
\begin{equation}\label{eq:F-bound}
 \|F_t\|_{L^\infty(\mathbb T^3)}\leq C=C(\lambda,\Lambda)\qquad \text{for all } t\geq 0.
\end{equation}

Using the idea of Loeper \cite[Proposition~5.1]{Lo05}, together with the
arguments in \cite[Proposition~3.3]{ACDF} and
\cite[Proposition~1]{ACDF14}, there exist
$\kappa=\kappa(\lambda,\Lambda)>0$ and
$C=C(\lambda,\Lambda)$ such that
\begin{equation}\label{eq:grad-time}
 \int_{\mathbb T^3}|D\partial_t  P_t^*|^{1+\kappa}\,\mathrm{d}x\leq C
 \qquad\text{for all }t\geq0.
\end{equation}
One can also obtain \eqref{eq:bd-rho_t}--\eqref{eq:grad-time} from  \cite[Theorem 15.12]{Le24}.

Recall that we normalize $P_t^*$ to have mean zero. Thus,
\begin{equation}\label{eq:mean-zero}
 \int_{\mathbb T^3}\partial_tP_t^*\,\mathrm{d}x=0\qquad \text{for all } t\geq 0.
\end{equation}
Poincar\'e's inequality and \eqref{eq:grad-time} then give
\begin{equation}\label{eq:Lp-P_t}
 \|\partial_tP_t^*\|_{L^{1+\kappa}(\mathbb T^3)}\leq C \qquad \text{for all } t\geq 0.
\end{equation}

We lift \eqref{eq:parP} periodically to $\mathbb R^3$. The periodic
normalization \eqref{eq:period} implies a uniform bound for
$D P_t^*(x)-x$ on a fundamental cell, and hence fixed-height sections
of $P_t^*$ have uniformly controlled diameter. Therefore the structural
constants in Theorem~\ref{thm:main} can be chosen uniformly in $t$ and in
the center of the section. Applying that theorem with
$p=1+\kappa$, using \eqref{eq:bd-rho_t}, \eqref{eq:F-bound}, and
\eqref{eq:Lp-P_t}, and covering a fundamental cell by finitely many
interior sections, we obtain some
$\gamma=\gamma(\lambda,\Lambda)\in(0,1)$ such that
\begin{equation}\label{eq:holder-Pstar}
 \|\partial_tP_t^*\|_{L^\infty((0,\infty);C^\gamma(\mathbb T^3))}
 \leq C.
\end{equation}

Finally, let $P_t$ be the Legendre transform of $P_t^*$. Differentiating the
Legendre identity at corresponding points gives
\begin{equation}\label{eq:legendre-time}
 \partial_tP_t(x)=-\partial_tP_t^*(D P_t(x)).
\end{equation}
The two-sided bounds \eqref{eq:bd-rho_t} and the periodic
Monge--Amp\`ere regularity theory give a uniform $C^{1,\beta}$ estimate for
$P_t$, for some $\beta=\beta(\lambda,\Lambda)\in(0,1)$. Combining this with
\eqref{eq:holder-Pstar} and \eqref{eq:legendre-time}, we obtain
\[
 |\partial_tP_t(x)-\partial_tP_t(y)|
 \leq C|D P_t(x)-D P_t(y)|^\gamma
 \leq C|x-y|^{\beta\gamma}\qquad \text{for all } t\geq 0.
\]
After decreasing the exponent if necessary, this proves
\eqref{eq:sg-main-intro}.
\end{proof}

%%%%%%%%%%%%%%%%%%%%%%%%%%%%%%%%%%%%%%%%%%

\subsection{Negative-norm stability and periodic homogenization} \label{sec:homogenization}
In this section, we prove Theorems~\ref{thm:MA-stability} and \ref{cor:MA-homogenization}.

Throughout this section, $n\geq2$ and $\Omega\subset\R^n$ is a fixed
bounded smooth uniformly convex domain.  Fix also
$\varphi\in C^3(\partial\Omega)$. For
$f\in L^\infty(\Omega)$ with $\lambda\leq f\leq\Lambda$, let $u_f$
be the convex Aleksandrov solution of
\[
 \det D^2u_f=f\quad\text{in }\Omega,
 \qquad u_f=\varphi\quad\text{on }\partial\Omega.
\]
Existence, uniqueness, and stability of these solutions are recalled
in \cite{Le24}.

\begin{proof}[Proof of Theorem \ref{thm:MA-stability}]
Fix $\F\in L^\infty(\Omega;\R^n)$ with
$\diver\F=f_1-f_0$. 

We show that there exist an exponent $\gamma\in (0, 1)$ and a constant $C>0$, 
depending only on $n,\lambda,\Lambda,\Omega$, and
$\|\varphi\|_{C^3(\partial\Omega)}$, such that
\begin{equation} 
\label{uFga}
\|u_1-u_0\|_{C^\gamma(\overline\Omega)}
 \leq C \|\F\|_{L^\infty(\Omega)}.\end{equation}

\smallskip\noindent
\emph{Step 1: Smooth densities.}
First assume that $f_0,f_1$ are smooth on $\overline\Omega$.

By the global $C^{2,\sigma}(\overline{\Omega})$ solvability of the Monge--Amp\`ere equation \cite{TW08}, the problems
\begin{equation}\label{eq:MA-homotopy}
 \det D^2u_t=(1-t)f_0+tf_1\quad\text{in }\Omega,
 \qquad u_t=\varphi\quad\text{on }\partial\Omega,
 \qquad 0\leq t\leq1,
\end{equation}
admit classical uniformly convex solutions $u_t\in C^{2,\sigma}(\overline{\Omega})$ for all $0<\sigma<1$. Note that the density stays between
$\lambda$ and $\Lambda$ and the boundary datum is the fixed function
$\varphi$.

\medskip

For each fixed $t$, the linearization of the Monge--Amp\`ere operator  at $u_t$
is an isomorphism from 
$C_0^{2,\sigma}(\overline\Omega):=\{w\in C^{2,\sigma}(\overline\Omega): w=0\quad\text{on }\p\Omega\}$ to $C^\sigma(\overline\Omega)$,
where $0<\sigma<1$; see \cite[Chapter~6]{GT01}.
The implicit-function theorem and uniqueness therefore show that
$t\mapsto u_t$ is $C^1$. Differentiating \eqref{eq:MA-homotopy},
we obtain
\[
 L_{u_t}v_t=\diver\F\quad\text{in }\Omega,
 \qquad v_t=0\quad\text{on }\partial\Omega
 \qquad\text{where } v_t=\partial_tu_t.
\]
For $n\geq3$, take $q=q_b+1$. Since
$\F\in L^\infty(\Omega)\subset L^q(\Omega)$,
Theorem~\ref{thm:global-finite-Lq}, applied with zero boundary datum,
gives an exponent $\gamma>0$ and a constant $C>0$, both uniform in $t$
and depending only on $n,\lambda,\Lambda,\Omega$, and
$\|\varphi\|_{C^3(\partial\Omega)}$, such that
\begin{equation}\label{eq:MA-linear-global-holder}
 \|v_t\|_{C^\gamma(\overline\Omega)}
 \leq C\|\F\|_{L^q(\Omega)}
 \leq C\|\F\|_{L^\infty(\Omega)},
 \qquad 0\leq t\leq1.
\end{equation}
In dimension two, the same estimate
follows from \cite[Theorem~1.2]{Le20}.

Since
\[
 u_1-u_0=\int_0^1v_t\,\mathrm{d}t,
\] 
integrating
 \eqref{eq:MA-linear-global-holder} (or its
two-dimensional analogue) establishes \eqref{uFga} for smooth
densities.

\medskip\noindent
\emph{Step 2: Bounded measurable densities.}
We use an approximation that preserves the divergence identity on
the fixed domain. After a translation, we can assume that $B_r(0)\subset\Omega$ for some $r>0$.
For $0<\delta<1/2$, let $a_\delta=1-\delta$ and
$s_\delta=\delta r/4$. The convexity of $\Omega$ then gives
$a_\delta\overline\Omega+B_{\delta r}(0)\subset\Omega$.
For a nonnegative standard mollifier $\eta_{s_\delta}$, define
\[
 f_{i,\delta}(x)
   =(f_i*\eta_{s_\delta})(a_\delta x),
 \qquad
 \F_\delta(x)
   =a_\delta^{-1}(\F*\eta_{s_\delta})(a_\delta x).
\]
The convolutions only sample points of $\Omega$. Hence the
approximations are smooth near $\overline\Omega$ and satisfy
\[
 \lambda\leq f_{i,\delta}\leq\Lambda,
 \qquad
 \diver\F_\delta=f_{1,\delta}-f_{0,\delta},
 \qquad
 \|\F_\delta\|_\infty\leq a_\delta^{-1}\|\F\|_\infty.
\]
Moreover, $f_{i,\delta}\to f_i$ in $L^1(\Omega)$ by continuity of
mollification and dilation in $L^1$.

Let $u_{i,\delta}\in C(\overline{\Omega})$ be the convex solution to
 \[\det D^2 u_{i,\delta}=f_{i,\delta}\quad\text{in }\Omega,\quad
 u_{i,\delta}=\varphi\quad\text{on }\partial\Omega.\]
From the stability of Aleksandrov solutions with fixed boundary data, the weak
continuity of Monge--Amp\`ere measures, and the uniqueness of the Monge--Amp\`ere equation, we find that
$u_{i,\delta}\to u_i$ uniformly on $\overline\Omega$; see \cite[Chapter 3]{Le24}.
Applying \eqref{uFga} to $u_{i,\delta}$ and $\F_\delta$, we have
\begin{equation*} 
\|u_{1,\delta}-u_{0,\delta}\|_{C^\gamma(\overline\Omega)}
 \leq C \|\F_\delta\|_{L^\infty(\Omega)}.\end{equation*}
Letting
$\delta\to0^+$, and using the uniform convergence and lower semicontinuity of the
H\"older norm, we obtain \eqref{uFga}.

Finally, taking the infimum over $\F$  in  \eqref{uFga} completes the proof of \eqref{eq:MA-stability}. The theorem is proved.
\end{proof}

\medskip

We are now ready to establish a linear convergence rate in periodic homogenization for the Monge--Amp\`ere equation.

\begin{proof}[Proof of Theorem~\ref{cor:MA-homogenization}]
Since $f-\bar f$ belongs to $L^\infty(\mathbb T^n)$ and has mean zero, the problem
\[
 \Delta_y\psi=f-\bar f\quad\text{on }\mathbb T^n,
 \qquad \int_{\mathbb T^n}\psi\, \mathrm{d}y=0
\]
has a unique solution $\psi\in W^{2, p}(\mathbb T^n)$ for all $p\in (1,\infty)$ with appropriate estimates;
see \cite[Theorem 9.15 and Lemma 9.17]{GT01}. In particular, for $p=2n$, we have
\[\|\psi\|_{W^{2, 2n}(\mathbb T^n)}\leq C(n)\|f-\bar f\|_{L^\infty(\mathbb T^n)}. \]
Let
\[G:=D_y\psi.\]
Then, the Sobolev embedding (see \cite[Section 7.7]{GT01})
gives
\[
 \diver_yG=f-\bar f\quad\text{on }\mathbb T^n,
 \qquad
 \|G\|_{L^\infty(\mathbb T^n)}
 \leq C(n)\|f-\bar f\|_{L^\infty(\mathbb T^n)}.
\]
Set \[\F_\varepsilon(x):=\varepsilon G(x/\varepsilon).\]
Then
\[
 \diver\F_\varepsilon=f(x/\varepsilon)-\bar f,
 \qquad
 \|\F_\varepsilon\|_{L^\infty(\Omega)}
 \leq C(n)\varepsilon\|f-\bar f\|_{L^\infty(\mathbb T^n)}.
\]
Clearly,
\[\|f(\cdot/\e)-\bar f\|_{\mathrm{div},\infty;\Omega} \leq \|\F_\varepsilon\|_{L^\infty(\Omega)}\leq C(n)\varepsilon\|f-\bar f\|_{L^\infty(\mathbb T^n)}.\]
 Since $\lambda\leq f,\bar f\leq\Lambda$ and both
$u_\varepsilon$ and $\bar u$ have the same boundary value $\varphi$,
Theorem~\ref{thm:MA-stability} proves both
\eqref{eq:hom-intro-uniform} and \eqref{eq:hom-intro-holder}. The theorem is proved.
\end{proof}

\end{document}